\documentclass{conm-p-l}
\usepackage{amssymb}
\font\smbf=ptmbo at 10.2pt

\def\bbR{\mathrm{I\!R}}

\def\rto{\bbR\nh^2}
\def\hatg{{\hat{g\hskip2pt}\hskip-1.3pt}}

\def\cro{\overline{\hskip-2pt\partial}}

\def\jg{\varPi}
\def\jq{Q}
\def\jp{B}
\def\js{S}
\def\jr{G}

\def\jz{Z}

\def\zb{\text{\smbf Z}}
\def\dzb{\dot{\hskip-1.5pt\text{\smbf Z}}\hskip1.5pt}

\def\hyp{\hskip.5pt\vbox
{\hbox{\vrule width2.5ptheight0.5ptdepth0pt}\vskip2pt}\hskip.5pt}
\def\hs{\hskip.7pt}
\def\hh{\hskip.4pt}
\def\nh{\hskip-.7pt}
\def\nnh{\hskip-1.5pt}
\def\hn{\hskip-.4pt}
\def\w{^{\phantom i}}
\def\txm{{T\hskip-2.9pt_x\w\hn M}}
\def\txr{{T\hskip-2.9pt_x\w\hn R}}
\def\tzm{{T\hskip-2.9pt_z\w\hn M}}

\def\dv{\delta}
\def\vg{\varGamma}
\def\ve{\varepsilon}

\def\vt{{\tau\hskip-4.55pt\iota\hskip.6pt}} %\uptau
\def\evt{{\tau\hskip-3.55pt\iota\hskip.6pt}} %\uptau
\def\svt{{\tau\hskip-3.45pt\iota\hskip.6pt}} %\uptau

\def\az{\alpha}
\def\bz{\beta}

\def\gz{\psi}
\def\tz{\theta}

\def\nz{n}
\def\kfo{\omega}
\def\rz{\hs\mathrm{r}\hs}
\def\sz{\sigma}

\def\bbC{{\mathchoice {\setbox0=\hbox{$\displaystyle\mathrm{C}$}
\hbox{\hbox to0pt{\kern0.4\wd0\vrule height0.9\ht0\hss}\box0}} 
{\setbox0=\hbox{$\textstyle\mathrm{C}$}\hbox{\hbox 
to0pt{\kern0.4\wd0\vrule height0.9\ht0\hss}\box0}} 
{\setbox0=\hbox{$\scriptstyle\mathrm{C}$}\hbox{\hbox 
to0pt{\kern0.4\wd0\vrule height0.9\ht0\hss}\box0}} 
{\setbox0=\hbox{$\scriptscriptstyle\mathrm{C}$}\hbox{\hbox 
to0pt{\kern0.4\wd0\vrule height0.9\ht0\hss}\box0}}}} 
\def\bbCP{\bbC\mathrm{P}}

\def\Lie{\pounds}

\def\tm{{T\hskip-.3ptM}}

\def\y{v}
\def\z{w}
\def\kfo{\omega}
\def\hatg{{\hat{g\hskip2pt}\hskip-1.3pt}}

\def\se{\mathrm{s}}

\newtheorem{theorem}{Theorem}[section]
\newtheorem{lemma}[theorem]{Lemma}

\theoremstyle{definition}

\theoremstyle{remark}
\newtheorem{remark}[theorem]{Remark}

\numberwithin{equation}{section}

\begin{document}

\title[Special Ric\-ci-Hess\-i\-an equations]{Special 
Ric\-ci\hh-\nh Hess\-i\-an equations on K\"ah\-ler manifolds}
%\title[Ric\-ci-Hess\-i\-an equations realized]{Ric\-ci\hh-\nh Hess\-i\-an
%equations realized: another proof}

\author[A. Derdzinski]{Andrzej Derdzinski}
\address{Department of Mathematics\\
The Ohio State University\hskip-1pt\\
231 W\hskip-2pt. 18th Avenue\\
Columbus, OH 43210, USA}
\email{andrzej@math.ohio-state.edu}
\author[P. Piccione]{Paolo Piccione}
\address{Departamento de Matem\'atica\\
Instituto de Matem\'atica e Esta\-t\'\i s\-ti\-ca\\
Uni\-ver\-si\-da\-de de S\~ao Paulo\\
Rua do Mat\~ao 1010, CEP 05508-900\\
S\~ao Paulo, SP, Brazil}
\email{piccione@ime.usp.br}

\thanks{Both authors' research was supported in part by a
FAPESP\nh-\hs OSU 2015 Regular Research Award (FAPESP grant: 2015/50265-6)}

%\date{Received: date / Accepted: date}
% The correct dates will be entered by the editor

\begin{abstract}Special Ricci-Hessian equations on K\"ahler manifolds $(M,g)$,
as defined by Maschler [Ann. Global Anal. Geom. 34 (2008), 367--380] involve 
functions $\tau$ on $M$ and state that, for some function $\alpha$ of the real
variable $\tau$, the sum of $\alpha\nabla d\tau$ and the Ricci tensor equals
a functional multiple of the metric $g$, while $\alpha\nabla d\tau$ itself is
assumed to be nonzero almost everywhere. Three well-known obvious ``standard''
cases are provided by 
(non-Einstein) gradient K\"ahler-Ricci solitons, conformally-Einstein K\"ahler
metrics, and special K\"ahler-Ricci potentials. We show that, outside of these
three cases, such an equation can only occur in complex dimension two and, at
generic points, it must then represent one of three %explicitly described
types, for which, up to normalizations, $\alpha=2\cot\tau$,
or $\alpha=2\coth\tau$, or $\alpha=2\tanh\tau$. We also use the
Car\-tan-K\"ah\-ler theorem to prove that these three types are actually
realized in a ``nonstandard'' way.
\end{abstract}
\makeatletter
\@namedef{subjclassname@2020}{\textup{2020} Mathematics Subject Classification}
\makeatother


\maketitle

\setcounter{section}{0}
\setcounter{theorem}{0}
\renewcommand{\thetheorem}{\Alph{theorem}}
\section*{Introduction}%\label{in}
\setcounter{equation}{0}
Following Maschler \cite[p.\ 367]{maschler}, one says that functions 
$\,\vt,\az,\sz\,$ on a Riemannian manifold $\,(M\nh,g)\,$ with the 
Ric\-ci tensor $\,\rz\,$ satisfy a {\it Ric\-ci-\nh Hess\-i\-an equation\/} if
\begin{equation}\label{fma}
\az\nabla\nh d\vt\,+\,\rz\,=\,\sz\hn g\quad\mathrm{for\ some\ function\
}\,\sz:M\to\bbR\hh,
\end{equation}
$\nabla$ being the Le\-vi-Ci\-vi\-ta connection of $\,g$. We call equation 
(\ref{fma}) {\it special\/} when
%Let us further assume that
\begin{equation}\label{anz}
\az\nabla\nh d\vt\ne0\,\mathrm{\ on\ a\ dense\ set,\ 
}\hh\dim M\nh=n>2\mathrm{,\ and\ 
}\,\az\,\mathrm{\ is\ a\ }C^\infty\,\mathrm{function\ of\ }\,\vt.
\end{equation}
Conditions (\ref{fma}) -- (\ref{anz}) are satisfied in several situations
that have been studied -- see below -- raising a natural question:
{\it Which functions\/ $\,\vt\mapsto\az\,$ can be realized in this way?} The
present paper provides an answer in the K\"ahler case, outside of the
classes that are already well understood. See Theorems~\ref{types} 
and \ref{reali}.
                                  
There are three well-known classes of examples leading to (\ref{fma}) --
(\ref{anz}).
\begin{enumerate}
  \def\theenumi{{\rm\roman{enumi}}}
\item[(I)] Non-Ein\-stein gradient Ric\-ci al\-most-sol\-i\-tons 
\cite{pigola-rigoli-rimoldi-setti,barros-ribeiro}, including (non-Ein\-stein)
gradient
Ric\-ci sol\-i\-tons \cite{hamilton}. Here $\,\az\,$ is a nonzero constant.
\item[(II)]  Con\-for\-mal\-ly-Ein\-stein metrics $\,g$, with $\,\vt>0\,$ and 
$\,\az=(\nz-2)/\vt$, the %con\-for\-mal\-ly-re\-lat\-ed
Ein\-stein metric 
being $\,\hatg=g/\vt^2\nh$. Cf.\ \cite[formula (6.2)]{derdzinski-maschler-03}.
\item[(III)]  Special K\"ah\-ler-Ric\-ci potentials $\,\vt\,$ on K\"ah\-ler
manifolds, at %$\,\mathrm{e}\hh$-ge\-ner\-ic 
points where $\,\mathrm{r}$ is not a multiple of $\,g$. See 
\cite[Remark 7.4]{derdzinski-maschler-03}.
\end{enumerate}
A {\it special K\"ah\-ler-Ric\-ci potential\/}
\cite[Sect.\,7]{derdzinski-maschler-03} 
on a K\"ah\-ler manifold $\,(M\nh,g)\,$ with the com\-plex-struc\-ture tensor
$\,J\,$ is any nonconstant function $\,\vt\,$ on $\,M$ having a 
real-hol\-o\-mor\-phic gradient 
$\,\y=\nabla\nh\vt\,$ %with the property
for which, at points where $\,\y\ne0$, all 
nonzero vectors orthogonal to $\,\y\,$ and $\,J\nh \y\,$ are eigen\-vec\-tors
of both $\,\nabla\nh d\vt$ and %the Ric\-ci tensor
$\,\mathrm{r}$. Such 
triples $\,(M\nh,g,\vt)$ are completely understood, both locally
\cite{derdzinski-maschler-03} and in the compact case 
\cite{derdzinski-maschler-06}. 

The classes (I) -- (III) are far from disjoint: for instance 
\cite[Corollary 9.3]{derdzinski-maschler-03}, 
%on the contrary, 
in the K\"ah\-ler category, if $\,\nz>4$, (II) is a special case of (III). 

We are interested in $\,M\nh,g,\vt,\az,\sz\,$ satisfying (\ref{fma}) --
(\ref{anz}) and such that
\begin{equation}\label{trv}
2\hh\mathrm{r}\hh(\y,\,\cdot\,)=-\hh dY\,\mathrm{\ for\
}\,\y=\nabla\nh\vt\,\mathrm{\ and\ }\,Y\nh=\Delta\hn\vt\hh.  
\end{equation}
Here, and throughout the paper, we use the notational conventions
\begin{equation}\label{ven}
\y=\nabla\nh\vt,%\hskip17ptu=J\y\hs,\hskip17pt\xi=g(u,\,\cdot\,)\hs,
\hskip17ptQ=g(\y,\y)\hs,\hskip17ptY\nh=\Delta\hn\vt,\hskip17pt\nz=\dim M
\end{equation}
whenever $\,(M\nh,g)\,$ is a Riemannian manifold and $\,\vt:M\to\bbR$.
As we point out near at the end of Section~\ref{pr}, with $\,J\,$ 
denoting the com\-plex-struc\-ture tensor,
\begin{equation}\label{imp}
\begin{array}{l}
\mathrm{for\ K}\ddot{\mathrm{a}}\mathrm{h\-ler\ metrics\ 
}\hs\,g\mathrm{,\ conditions\ \hs(\ref{fma})\hn-\hn(\ref{anz})\hs\ imply\
(\ref{trv}),}\\
\mathrm{and\ the\  gradient\ }\y=\nabla\nh\vt\mathrm{\ is\ a\ 
real}\hyp\mathrm{hol\-o\-mor\-phic\ vector\ field}\\
\mathrm{or,\hh\ equivalently,\hh\ }\,J\y\,\mathrm{\hh\ is\hh\ a\hh\ 
real}\hyp\mathrm{hol\-o\-mor\-phic\hh\ }\,g\hyp\mathrm{Kil\-ling\hh\ field.}
\end{array}
\end{equation}
Assuming (\ref{fma}) -- (\ref{anz}), we may treat the derivatives
$\,\az'\nh=\hs d\az/d\vt\,$ and $\,\az''$ both as functions of the real
variable $\,\vt\,$ and as functions $\,M\to\bbR$. In Sections~\ref{rh}
and~\ref{km} we prove the following two results, as well as
Theorem~\ref{types}, stated below.
\begin{theorem}\label{apppa}Under the hypotheses\/ {\rm(\ref{fma})}
--\/ {\rm(\ref{trv})}, at points where\/ $\,\az''\nh+\az\az'\nh\ne0$ and\/ 
$\,d\vt\ne0$, both\/ $\,Q=g(\nabla\nh\vt,\nnh\nabla\nh\vt)\,$ and\/
$\,Y\nh=\Delta\hn\vt\,$ are, locally, functions of\/ $\,\vt$.
\end{theorem}
\begin{theorem}\label{spkrp}Let functions\/ $\,\vt,\az,\sz\,$ satisfy a
special Ric\-ci-\nh Hess\-i\-an equation\/ {\rm(\ref{fma})}, with\/
{\rm(\ref{anz})}, on a K\"ah\-ler manifold\/ $\,(M\nh,g)\,$ of real
dimension\/ $\,\nz\ge4$. 
If\/ $\,\az\,d\az$ and\/ $\,d\vt\,$ are nonzero at all points of an open
sub\-man\-i\-fold\/ $\,\,U\hs$ of\/ $\,M\nh$, and
\begin{enumerate}
  \def\theenumi{{\rm\roman{enumi}}}
\item $\nz>4$, or
\item $\nz=4\,$ and\/ $\,d\sz\wedge d\vt=0\,$ identically in\/ $\,\,U\hs$ or,
finally, 
\item $dQ\wedge d\vt=0\,$ everywhere in\/ $\,\,U\nh$, where\/ 
$\,Q=g(\nabla\nh\vt,\nnh\nabla\nh\vt)$,
\end{enumerate}
then\/ $\,\vt:U\to\bbR\,$ is a special K\"ah\-ler-\nh Ric\-ci potential on the 
K\"ah\-ler manifold\/ $\,(U\nh,g)$.
\end{theorem}
With $\,\y,Q,Y\hs$ as in (\ref{ven}), a function $\,\vt\,$ on a 
Riemannian manifold $\,(M\nh,g)\,$ has $\,dQ\wedge d\vt=0\,$
if and only if $\,Q\,$ is locally, at points where $\,d\vt\ne0$, a function 
of $\,\vt$. This amounts 
to requiring the integral curves of $\,\y\,$ to be re\-pa\-ram\-e\-trized
geodesics (since, due to formula (\ref{nvv}) below, the latter condition means
that $\,\nabla\nh\!_\y\w\y\,$ is a functional multiple of $\,\y$). Such
functions $\,\vt$, called {\it trans\-nor\-mal}, have been studied extensively
\cite{wang,miyaoka,bolton}, and are referred to as {\it
iso\-par\-a\-met\-ric\/} when, in addition, $\,dY\nnh\wedge d\vt=0$.

Theorem~\ref{spkrp} renders the trans\-nor\-mal case $\,dQ\wedge d\vt=0$, as 
well as real dimensions $\,\nz>4$, rather uninteresting in the context
of special Ric\-ci-\nh Hess\-i\-an equations (\ref{fma}) -- (\ref{anz}) on
K\"ah\-ler manifolds, since at $\,d\az$-ge\-ner\-ic points (see the end of
Section~\ref{pr}) one then ends up with examples (I) or (III) above, cf.\
Remark~\ref{ctsgm}, of
which the former is the subject of a large existing literature, and the 
latter, as mentioned earlier, has been completely described. This is why our
next two results focus exclusively on the real dimension four and functions
$\,\vt\,$ with $\,dQ\wedge d\vt\,$ not identically zero.
\begin{remark}\label{affmo}Equation (\ref{fma}), with (\ref{anz}), remains 
satisfied after $\,\vt\,$ and the function $\,\vt\mapsto\az=\az(\vt)\,$ have 
been subjected to an {\it af\-fine modification\/} in the sense of being 
replaced with $\,\hat\vt$ and $\,\hat\vt\mapsto\hat\az(\hat\vt)\,$ given by
$\,\hat\vt=p+\vt/c\,$ and $\,\hat\az(\hat\vt)
=c\hh\az(c\hh\hat\vt-cp)$ for real constants $\,c\ne0\,$ and $\,p$.
\end{remark}
%Here is how Theorem~\ref{spkrp} reduces the
%problem of describing the remaining special Ric\-ci-\nh Hess\-i\-an equations
%(\ref{fma}) -- (\ref{anz}) on K\"ah\-ler manifolds:
\begin{theorem}\label{types}If the special Ric\-ci-\nh Hess\-i\-an
equation\/ {\rm(\ref{fma})} and\/ {\rm(\ref{anz})} both hold for functions\/
$\,\vt,\az,\sz\,$ on a K\"ah\-ler manifold\/ $\,(M\nh,g)\,$ of real
dimension four, while\/ $\,dQ\wedge d\vt\ne0\,$ everywhere in an open connected
set\/ $\,\,U\subseteq M\nh$, then the function\/ $\,\az\,$ of the variable\/
$\,\vt\,$ and its derivative\/
$\,\az'\nh=\hs d\az/d\vt\,$ satisfy, on\/ $\,\,U\nh$, the equation
\begin{equation}\label{app}
\az''\hh+\,\az\az'\hh=\,0\hh,\mathrm{\ that\ is,\ %which\ amounts\ to\ }\,
}\,2\az'\nnh+\hs\az^2=\,4\ve\,\mathrm{\ with\ a\ constant\ }\,\ve\in\bbR\hh.
\end{equation}
In addition, for\/ $\,Q\,$ and\/ $\,Y\nnh$ as in\/ {\rm(\ref{ven})},
the functions
\begin{equation}\label{cts}
2\hh\tz\,=\,\az\hs\se+4\ve Y\,\,\mathrm{\ and\ }\,\,
\kappa\,=\,\tz\gz\,+\,\az^{-\nnh1}Y\hs-\hs\,Q\hskip7pt\mathrm{are\ both\
constant,}
\end{equation}
$\gz\,$ being given by\/ 
$\,4\ve\hh\gz=\vt-2/\nh\az$, if\/ $\,\ve\ne0$, 
or\/ $\,3\gz=2/\nh\az^3\nh$, when\/ $\,\ve=0$. Furthermore, $\,\sz\,$ in\/ 
{\rm(\ref{fma})} and the function\/ $\,F\,$ of the variable\/ $\,\vt\,$
characterized by
\begin{equation}\label{fcf}
4\ve F=\tz(2-\vt\az)+4\ve\kappa\az\,\mathrm{\ for\ }\,\ve\ne0\mathrm{,\ and\ 
}\,F=\kappa\az-2\hh\tz/(3\az^2)\,\mathrm{\ if\ }\,\ve=0\hh,
\end{equation}
and thus depending on the real constants $\,\tz,\kappa$, satisfy the conditions
\begin{equation}\label{cns}
\mathrm{a)}\hskip6ptY\nnh-Q\az=F,\quad
\mathrm{b)}\hskip6pt2\sz=-(Q\az'\nh+F'),\quad
\mathrm{c)}\hskip6pt\Delta\az=F\nh\az'\nh=-\nh F''.
\end{equation}
Finally, up to af\-fine 
modifications -- see Remark\/~{\rm\ref{affmo}} -- the pair\/
$\,(\az(\vt),\ve)\,$ is one of the following five\/{\rm:} 
$\,(2,1),\,\,(2/\nh\vt,0),\,\,(2\hn\tanh\vt,1),\,\,(2\hn\coth\vt,1),\,\,(2\hn\cot\vt,-\nnh1)$.
\end{theorem}
\begin{theorem}\label{reali}Each of the five options listed in
Theorem\/~{\rm\ref{types}}, namely,
\[
(2,1)\hh,\quad(2/\nh\vt,0)\hh,\quad(2\hn\tanh\vt,1)\hh,\quad(2\hn\coth\vt,1)\hh,\quad(2\hn\cot\vt,-\nnh1)\hh,
\]
is realized by a special
Ric\-ci-\nh Hess\-i\-an equation\/ {\rm(\ref{fma})} -- {\rm(\ref{anz})} on a
real-an\-a\-lyt\-ic K\"ah\-ler manifold\/ $\,(M\nh,g)\,$ of real dimension
four such that, with\/ $\,\y=\nabla\nh\vt\,$ and\/ $\,Q=g(\y,\y)$, one has\/
$\,dQ\wedge d\vt\ne0\,$
somewhere in\/ $\,M\,$ and\/ $\,J\y\,$ lies in a two\hs-di\-men\-sion\-al
Abel\-i\-an Lie algebra of Kil\-ling fields.

For\/ $\,(2,1)\,$ and\/ $\,(2/\nh\vt,0)\,$ one can choose\/
$\,(M\nh,g)\,$ to
be compact and bi\-hol\-o\-mor\-phic to the two-point blow-up of\/ 
$\,\bbCP^2\nh$, with\/ $\,g\,$ which is
the Wang\hs-\hn Zhu toric \hbox{K\"ah\-ler\hs-} Ric\-ci sol\-i\-ton\/
{\rm\cite[Theorem 1.1]{wang-zhu}} or, respectively, 
the Chen-LeBrun-\nh Weber con\-for\-mal\-ly\hs-\hn Ein\-stein K\"ahler
metric\/ {\rm\cite[Theorem A]{chen-lebrun-weber}}.
\end{theorem}
In contrast with the final clause of Theorem~\ref{reali}, we do not know
whether the remaining three options, $\,(2\hn\tanh\vt,1)$,
$\,(2\hn\coth\vt,1)\,$ and $\,(2\hn\cot\vt,-\nnh1)$, may be realized on a {\it
compact\/} K\"ah\-ler surface. An an\-a\-lyt\-ic-con\-tin\-u\-a\-tion
phenomenon described below (Section~\ref{ac}) may hint at plausibility of 
trying to obtain such compact examples via small 
deformations of the Wang-Zhu or Chen-LeBrun-Weber metric, combined with
suitable af\-fine modifications.

For the pairs $\,(2,1)\,$ and $\,(2/\nh\vt,0)\,$ in Theorem~\ref{types}, the 
constancy conclusions of (\ref{cts}) are well known: \cite[p.\,201]{chow}, 
\cite[p.\,417,\,Prop.\,3(i) and p.\,419,\,formula\,(40)]{derdzinski-83}.

The paper is organized as follows. Section~\ref{pr} contains the
preliminaries. Consequences of special 
Ric\-ci-\nh Hess\-i\-an equations, leading to proofs of Theorems~\ref{apppa},
\ref{spkrp} and~\ref{types}, are  presented in the next two sections.
Sections~\ref{lk} through~\ref{pt} are devoted to proving Theorem~\ref{reali}:
we rephrase it as solvability of the system (\ref{foe}) of qua\-si-lin\-e\-ar
first-or\-der partial differential equations, subject to the additional
conditions (\ref{det}), 
which allows us to derive our claim from the Car\-tan-K\"ah\-ler
theorem for exterior differential systems.

\renewcommand{\thetheorem}{\thesection.\arabic{theorem}}
\section{Preliminaries}\label{pr}
\setcounter{equation}{0}
All manifolds and Riemannian metrics are assumed to be of class 
$\,C^\infty\nnh$. By definition, a manifold is connected. We use the symbol
$\,\dv\,$ for divergence.

On a manifold with a tor\-sion-free connection $\,\nabla\nh$, the Ric\-ci 
tensor $\,\mathrm{r}\,$ satisfies the Boch\-ner identity 
$\,\mathrm{r}(\,\cdot\,,\y)=\hs\dv\hskip1.2pt\nabla\nh \y
-\hs d\hs[\dv\hs \y]$, where $\,\y\,$ is any vector field. 
Its coordinate form $\,R\nh_{jk}\w \y^{\hs k}\nh=\y^{\hs k}{}_{,\hs jk}\w\nh
-\y^{\hs k}{}_{,\hs kj}\w$ arises via contraction from the Ric\-ci identity 
$\,\y^{\hs l}{}_{,\hs jk}\w\nh
-\y^{\hs l}{}_{,\hs kj}\w\nh=R\nh_{jkq}\w{}^l\hs \y\hh^q\nh$. (We use the sign 
convention for $\,R\,$ such that 
$\,R\nh_{jk}\w\nh=R\nh_{jqk}\w{}^q\nh$.) Applied to the gradient $\,\y\,$ of a 
function $\,\vt\,$ on a Riemannian manifold, this yields
\begin{equation}\label{rcd}
\dv\hh[\nabla\nh d\vt]\,=\,\mathrm{r}\hs(\y,\,\cdot\,)\,\,
+\hs\,dY\nh,\hskip12pt\mathrm{\ with\ }\,\,\y=\nabla\nh\vt\,\,\mathrm{\ and\ 
}\,\,Y\nh=\Delta\hn\vt.
\end{equation}
On the other hand, given a function $\,\vt\,$ on a Riemannian manifold,
\begin{equation}\label{nvv}
2[\nabla\nh d\vt](\y,\,\cdot\,)\,=\,dQ\nh,\hskip12pt\mathrm{\ where\ 
}\,\,\y=\nabla\nh\vt\,\,\mathrm{\ and\ }\,\,Q=g(\y,\y)\hh,
\end{equation}
as one sees noting that, in local coordinates,
$\,(\vt\nnh_{\nh,\hs k}\w\vt^{\hh,\hh k})\nh_{,\hh j}\w
=2\vt\nnh_{\nh,\hs kj}\w\vt^{\hh,\hh k}$. We can rewrite the relations
(\ref{rcd})
-- (\ref{nvv}) using the interior product $\,\imath_\y\w$, and then they read
\begin{equation}\label{rwr}
\mathrm{a)}\hskip6pt\dv\hh[\nabla\nh d\vt]\,=\,\imath_\y\w\mathrm{r}\hs\,
+\hs\,dY\nh,\hskip12pt\mathrm{b)}\hskip6pt
2\imath_\y\w[\nabla\nh d\vt]\,=\,dQ\hh,\hskip12pt\mathrm{\ with\ (\ref{ven}).}
\end{equation}
Finally, for the Ric\-ci tensor 
$\,\mathrm{r}\,$ and scalar curvature $\,\se\,$ of any Riemannian metric,
\begin{equation}\label{brt}
2\hs\delta\mathrm{r}\,=\,d\hh\se\hh,
\end{equation}
which is known as the Bian\-chi identity for the Ric\-ci tensor. Its
coordinate form $\,2g^{kl}R\nh_{j\nh k,l}\w=s_{\nh,\hs j}\w$ %of which
is immediate if one transvects with (``multiplies'' by) $\,g^{kl}$ the
equality $\,R\nh_{j\nh kl}\w{}^q{}_{,\hs q}\w
=R\nh_{jl,k}\w-R\nh_{kl,j}\w$ obtained by contracting the second Bian\-chi 
identity.

The {\it har\-mon\-ic-flow condition\/} for a vector field $\,\y\,$ on a 
Riemannian manifold $\,(M\nh,g)$, meaning that the flow of $\,\y\,$ consists of
(local) harmonic dif\-feo\-mor\-phisms, is known \cite{nouhaud} to be
equivalent to the equation
\begin{equation}\label{hfc}
g(\Delta\y,\,\cdot\,)\,=\,-\mathrm{r}\hh(\y,\,\cdot\,)
\end{equation}
the vector field $\,\Delta\y\,$ having the local components
$\,[\Delta\y]^j\nh=\y^{j,k}{}_k\w$. See also
\cite[Theorem 3.1]{dodson-trinidad-perez-vazquez-abal}. When
$\,\y=\nabla\nh\vt\,$ is the gradient of a function $\,\vt:M\to\bbR$,
\begin{equation}\label{hfa}
\mathrm{the\ har\-mon\-ic}\hyp\mathrm{flow\ condition\ (\ref{hfc})\ amounts\
to\ (\ref{trv}).}
\end{equation}
In fact, by (\ref{rcd}),
$\,2\mathrm{r}\hs(\y,\,\cdot\,)+\hs dY\nh=\dv\hh[\nabla\nh d\vt]
+\mathrm{r}\hs(\y,\,\cdot\,)=g(\Delta\y,\,\cdot\,)
+\mathrm{r}\hh(\y,\,\cdot\,)$,
as $\,[\Delta\y]_j\w=\y_{j,k}\w{}^k\nh=\vt\nnh_{,jk}\w{}^k\nh
=\vt\nnh_{,kj}\w{}^k\nh=\vt^{,kj}{}_k\w=(\dv\hh[\nabla\nh d\vt])\hn_j\w$.

On the other hand -- see, e.g., \cite[Lemma 5.2]{derdzinski-maschler-03} -- 
on a K\"ah\-ler manifold $\,(M\nh,g)$,
\begin{equation}\label{rho}
\begin{array}{l}
\mathrm{conditions\ (\ref{fma})\hskip-2.6pt-\hskip-2.6pt(\ref{anz})\ imply\ 
real}\hyp\mathrm{hol\-o\-mor\-phic\-i\-ty\ of\ the}\\
\mathrm{gradient\ }\y\hn=\nnh\nabla\nh\vt\mathrm{,\hskip1ptwhile\
}J\hn\y\mathrm{\ is\ then\ a\ hol\-o\-mor\-phic\ Kil\-ling}\\
\mathrm{field,\ due\ to\ the\ resulting\ Her\-mit\-i\-an\ symmetry\ of\
}\,\nabla\nh d\vt\hh.
\end{array}
\end{equation}
Since 
hol\-o\-mor\-phic mappings between
K\"ah\-ler manifolds are harmonic, every hol\-o\-mor\-phic vector field on a
K\"ah\-ler manifold satisfies (\ref{hfc}), cf.\ 
\cite[Remark 3.2]{dodson-trinidad-perez-vazquez-abal}. Now (\ref{imp})
follows from (\ref{hfa}). In other words, as observed by Calabi
\cite{calabi}, on K\"ah\-ler manifolds one has
\begin{equation}\label{rhg}
\mathrm{equation\ (\ref{trv}),\ with\ (\ref{ven}),\ for\ all\ 
real}\hyp\mathrm{hol\-o\-mor\-phic\ gradients\ }\,\y=\nabla\nh\vt\hh.
\end{equation}
Given a tensor field $\,\theta\,$ on a manifold 
$\,M\nh$, we say that a point $\,x\in M\,$ is $\,\theta${\it-ge\-ner\-ic\/} 
if $\,x\,$ has a neighborhood on which either $\,\theta=0$ identically, or 
$\,\theta\ne0\,$ everywhere. Such points clearly form a 
dense open sub\-set of $\,M\nh$.

\section{Ric\-ci\hh-\nh Hess\-i\-an equations}\label{rh}
\setcounter{equation}{0}
As a consequence of (\ref{fma}), for the scalar curvature $\,\se$, with
(\ref{ven}),
\begin{equation}\label{nec}
\nz\hh\sz\,=\,Y\hskip-2.3pt\az\,+\,\hs\se\hh,\mathrm{\ \ where\ \ }\,\nz=\dim M.
\end{equation}
Applying $\,2\hh\imath_\y\w$ or $\,2\hh\dv\,$ to %both sides of 
(\ref{fma}), we obtain, from (\ref{rwr}) -- (\ref{brt}) and (\ref{ven}),
\begin{equation}\label{adq}
\begin{array}{rl}
\mathrm{i)}&
\az\hs dQ\,+\,2\hh\mathrm{r}\hh(\y,\,\cdot\,)\,=\,2\sz\hs d\vt\hh,\\
\mathrm{ii)}&
2[\nabla\nh d\vt](\nabla\nnh\az,\,\cdot\,)\,
+\,2\az\hs[\mathrm{r}\hh(\y,\,\cdot\,)\,+\hs\,dY]\,+\hs\,d\hs\se\,
=\,2\hs d\sz\hh.
\end{array}
\end{equation}
In the case where (\ref{fma}) -- (\ref{anz}) hold along with (\ref{trv}), one
may rewrite (\ref{adq}) as
\begin{equation}\label{dqy}
\begin{array}{rl}
\mathrm{i)}&
\az\hs dQ\,\hh-\,dY\hs=\,2\sz\hs d\vt\hh,\\
\mathrm{ii)}&
2[\nabla\nh d\vt](\nabla\nnh\az,\,\cdot\,)\,
+\,\az\hs dY\hs+\hs\,d\hs\se\,=\,2\hs d\sz\hh,
\end{array}
\end{equation}
which, in view of (\ref{nvv}) and (\ref{nec}), amounts to nothing else than
\begin{equation}\label{dqm}
\begin{array}{rl}
\mathrm{i)}&
d(Q\az\,-\,Y)\,=\,(Q\az'\hs+\,2\sz)\hs d\vt\hh,\\
\mathrm{ii)}&
d\hh[Q\az'\hs+\,(\nz-2)\hh\sz]\,=\,(Q\az''\hs+\,Y\hskip-2.3pt\az')\hs d\vt\hh,
\end{array}
\end{equation}
as the assumption, in (\ref{anz}), that $\,\az\,$ is a 
$\,C^\infty\nnh$ function of $\,\vt\,$ allows us to write
\begin{equation}\label{ape}
d\az\,=\,\az'\nh d\vt\hh,\quad\nabla\nnh\az\,
=\,\az'\nh\y\hh,\quad2[\nabla\nh d\vt](\nabla\nnh\az,\,\cdot\,)\,
=\az'\hn dQ\hh,\quad\mathrm{where\ \ }\,\az'\nh=\hs d\az/d\vt\hh,
\end{equation}
since (\ref{nvv}) gives 
$\,2[\nabla\nh d\vt](\nabla\nnh\az,\,\cdot\,)
=2\az'[\nabla\nh d\vt](\y,\,\cdot\,)=\az'\hn dQ$. 
Due to (\ref{dqm}), conditions (\ref{fma}) -- (\ref{trv}) imply that, 
locally, at points at which $\,d\vt\ne0$,
\begin{equation}\label{fot}
\begin{array}{l}
Q\az\hs-\hs Y\hs\mathrm{\ and\ }\hs\,Q\az'+\hs(\nz-2)\hh\sz\,\mathrm{\ are\ 
functions\ of\ }\,\vt\mathrm{,\ with}\\
\mathrm{the\ respective\ }\,\vt\hyp\mathrm{derivatives\
}\,Q\az'+\hs2\sz\,\mathrm{\ and\ }\,Q\az''+\,Y\hskip-2.3pt\az'\\
\mathrm{which,\nnh\ consequently\nnh,\nnh\ must\hn\ themselves\hn\ be\hn\
functions\hn\ of\hn\ }\hs\vt\nh.
\end{array}
\end{equation}
\begin{proof}[Proof of Theorem~\ref{apppa}]At the points in question, using
(\ref{fot}) to equate both $\,Q\az-Y\hs$ and $\,Q\az''\nh+\,Y\hskip-2.3pt\az'$ to some 
specific functions of $\,\vt$, we obtain a system of two linear equations with 
the nonzero determinant $\,\az''\nh+\az\az'\nh$, imposed on the unknowns
$\,Q,\nh Y\nh$, and our assertion follows since $\,\az''\nh+\az\az'$ is also 
a function of $\,\vt$.
\end{proof}
Assuming only (\ref{fma}), for $\,\nz=\dim M\nh$, with the aid of (\ref{nec})
we rewrite (\ref{adq}) as
\[%\begin{equation}\label{nad}
\begin{array}{l}
\nz\hh[\az\hs dQ\,+\,2\hh\mathrm{r}\hh(\y,\,\cdot\,)]\,
-\,2(Y\hskip-2.3pt\az\,+\,\hs\se)\hs d\vt\,=\,0\hh,\\
2\nz\hh\{\nh[\nabla\nh d\vt](\nabla\nnh\az,\,\cdot\,)
+\nh\az\hh\mathrm{r}\hh(\y,\,\cdot\,)\}\nh
+\nh2[(\nz-1)\az\hs dY\nnh-\nh Y\nh d\az]\nh+\nh(\nz-2)\hs d\hh\se=0\hh,
\end{array}
\]%\end{equation}
If (\ref{trv}) holds as well, replacing $\,2\hh\mathrm{r}\hh(\y,\,\cdot\,)\,$
here with $\,-\hh dY\hs$ we obtain 
$\,\nz\hh(\az\hs dQ-\hs dY)
-2(Y\hskip-2.3pt\az+\hs\se)\hs d\vt=0\,$ and 
$\,2\nz\hh[\nabla\nh d\vt](\nabla\nnh\az,\,\cdot\,)\nh+(\nz-2)(\az\hs dY\nnh
+d\hh\se)-\nh2Y\nh d\az=0$. 
Thus, when (\ref{fma}) -- (\ref{trv}) are all satisfied, (\ref{ape}) gives
\begin{equation}\label{apd}
\begin{array}{rl}
\mathrm{a)}&
\nz\hh(\az\hs dQ\,-\hs\,dY)\,
-\,2(Y\hskip-2.3pt\az\,+\,\hs\se)\hs d\vt\,=\,0\hh,\\
\mathrm{b)}&
\nz\hh\az'\nh dQ\,
+\,(\nz-2)(\az\hs dY\hs+\,\hs d\hh\se)\,-\,2Y\hskip-2.3pt\az'\nh d\vt\,
=\,0\hh.
\end{array}
\end{equation}

\section{Ric\-ci\hh-\nh Hess\-i\-an equations on K\"ah\-ler 
manifolds}\label{km}
\setcounter{equation}{0}
The goal of this section is to prove Theorems~\ref{spkrp} and~\ref{types}.

In any complex manifold, $\,d\hh\omega=0\,$ and 
$\,\omega\hh(J\,\cdot\,,\,\cdot\,)\,$ is symmetric if 
$\,\omega=i\hh\partial\overline{\partial}\hs\vt$, that is, if 
$\,2\hs\omega=-\hh d\hskip1pt[\hn J^*\nnh d\vt]\,$ for a 
real-val\-ued function $\,\vt$, with the $\,1$-form 
$\,J^*\nnh d\vt=(d\vt)\hn J$  
which sends any tangent vector field $\,\y\,$ to $\,d\nnh_{J\nh\y}\w\vt$. 
Our ex\-te\-ri\-or-de\-riv\-a\-tive and 
ex\-te\-ri\-or-prod\-uct conventions, for $\,1$-forms 
$\,\xi,\xi'\,$ and vector fields $\,u,\y$, are
\begin{equation}\label{wdg}
\begin{array}{l}
[d\hh\xi](u,\y)\,=\,d_u\w[\hh\xi(\y)]\,-\,d_\y\w[\hh\xi(u)]\,-\,\xi([u,\y])\hh,\\
{}[\xi\wedge\xi'](u,\y)\,=\,\xi(u)\xi'(\y)\,-\,\xi(\y)\xi'(u)\hh.
\end{array}
\end{equation}
For a tor\-sion\-free connection $\,\nabla\nh$, (\ref{wdg}) gives
$\,[d\hh\xi](u,\y)=[\nabla\nh\!_u\w\xi](\y)-[\nabla\nh\!_\y\w\xi](u)$, so
that, if in addition $\,\nabla\nnh J=0$, on an al\-most-com\-plex manifold,
\begin{equation}\label{idd}
2i\hh\partial\overline{\partial}\hs\vt\hs\,
=\,\hs[\nabla\nh d\vt](J\,\cdot\,,\,\cdot\,)\hs\,
-\,\hs[\nabla\nh d\vt](\,\cdot\,,J\,\cdot\,)\hh.
\end{equation}
In the case of a K\"ah\-ler metric $\,g\,$ on a complex manifold $\,M\nh$, 
(\ref{fma}) implies that
\begin{equation}\label{add}
i\hh\az\hskip1pt\partial\overline{\partial}\hs\vt\hs\,
+\,\rho\hs\,=\,\sz\kfo\hh,
\end{equation}
$\kfo,\rho\,$ being the K\"ah\-ler and Ric\-ci forms,
with both terms on the right-hand side of (\ref{idd}) equal, as 
Her\-mit\-i\-an symmetry of $\,\nabla\nh d\vt\,$ follows %via (\ref{fma})
from those of $\,\rho\,$ and $\,\kfo$.
\begin{remark}\label{eqhol}
As an obvious consequence of the last line in (\ref{rho}), 
if $\,g\,$ is a K\"ah\-ler metric, conditions
(\ref{fma}) -- (\ref{anz}) are 
{\it equivalent\/} to (\ref{add}) along with (\ref{anz}) and 
real-hol\-o\-mor\-phic\-i\-ty of the gradient $\,\y=\nabla\nh\vt$.
\end{remark}
\begin{remark}\label{injct}For the K\"ah\-ler form $\,\kfo\,$ of a K\"ah\-ler 
manifold $\,(M\nh,g)\,$ of real dimension $\,\nz\ge4$, the operator 
$\,\zeta\mapsto\zeta\wedge\hskip1pt\kfo\,$ acting on differential $\,q$-forms
is injective if $\,q=2\,$ and $\,\nz>4$, or $\,q=1$. Namely, the contraction
of $\,\zeta\wedge\hskip1pt\kfo\,$ against $\,\kfo\,$ yields a nonzero constant
times $\,(\nz-4)\hs\zeta+2\langle\kfo,\zeta\rangle\hh\kfo\,$ (if
$\,q=2$), or times $\,(\nz-2)\hs\zeta\,$ (if $\,q=1$). In the former case,
$\,\zeta\,$ with $\,\zeta\wedge\hskip1pt\kfo=0\,$ is thus a multiple
of $\,\kfo$, and hence $\,0$.  
\end{remark}
\begin{remark}\label{ctsgm}Whenever (\ref{add}) with a {\it constant\/}
$\,\az\,$ holds on a K\"ah\-ler manifold of real dimension $\,\nz\ge4$,
constancy of $\,\sz\,$ follows (from Remark~\ref{injct}, as
$\,d\sz\wedge\hskip1pt\kfo=0$).
\end{remark}
We have the following result, due to Maschler \cite[Proposition 3.3]{maschler}.
\begin{lemma}\label{dadsz}Condition\/ {\rm(\ref{fma})} on a K\"ah\-ler 
manifold\/ $\,(M\nh,g)\,$ of real dimension\/ $\,\nz>4\,$ implies that\/ 
$\,d\sz\wedge d\az=0$. Equivalently, wherever\/ $\,d\az\,$ is nonzero, 
$\,\sz$ must, locally, be a function of\/ $\,\az$.
\end{lemma}
\begin{proof}By (\ref{add}), $\,0=d\rho=d\sz\wedge\hskip1pt\kfo
-d\az\wedge i\hh\hskip1pt\partial\overline{\partial}\hs\vt$, so that 
$\,d\az\wedge d\sz\wedge\hskip1pt\kfo=0$, and our assertion is immediate from 
Remark~\ref{injct}.
\end{proof}
\begin{proof}[Proof of Theorem~\ref{spkrp}]In all three cases (i) -- (iii),
$\,d\sz\wedge d\vt=0$. For (i), this follows from Lemma~\ref{dadsz} while, 
when $\,dQ\wedge d\vt=0\,$ on 
$\,\,U\nh$, we see that, in view of the equality 
$\,\az'\nh dQ+\az\hs dY\nh=d(2\sz-\se)\,$ arising from (\ref{dqy}.ii) and
(\ref{ape}), $\,Y$ and $\,2\sz-\se$ are, 
locally, functions of $\,\vt$, and hence so is $\,\sz$, as a consequence of
(\ref{nec}) with $\,\nz\ge4$. 
Now \cite[Corollary 9.2]{derdzinski-maschler-03} yields our claim.
\end{proof}
\begin{proof}[Proof of Theorem~\ref{types}]As $\,dQ\wedge d\vt\ne0\,$
everywhere in $\,\,U\nh$, Theorem~\ref{apppa} implies (\ref{app}) and,
consequently, also the final clause about the five possible pairs.

Next, in (\ref{app}),
$\,4\hs d\tz=2\hs d\hs[\az\hs\se+(2\az'\nh+\az^2)Y]
=2[\az\hs d\hh\se+\se\hskip1.2ptd\az+(2\az'\nh+\az^2)dY]$ which, as 
$\,\nz=4$, equals, in view of (\ref{ape}),
\[
\az[\nz\hh\az'\nh dQ+(\nz-2)(\az\hs dY\nh+\hs d\hh\se)
-2Y\hskip-2.3pt\az'\nh d\vt]\,
-\,\az'[\nz\hh(\az\hs dQ-\hh dY)-2(Y\hskip-2.3pt\az+\hs\se)\hs d\vt]\hh,
\]
and hence vanishes due to (\ref{apd}). On the other hand, the 
function $\,\gz\,$ of $\,\vt\,$ defined in the theorem is an antiderivative of
$\,1/\az^2\nh$, meaning that
\begin{equation}\label{ant}
\gz'\hs=\,1/\az^2.
\end{equation}
Namely, by (\ref{app}). 
$\,4\ve\hh\gz'\nh=1+2\az'\nh/\nh\az^2\nh=(2\az'\nh+\az^2)/\nh\az^2\nh
=4\ve/\nh\az^2$ if $\,\ve\ne0$, and 
$\,3\gz'\nh=-6\az'\nh/\nh\az^4\nh=3/\nh\az^2$ when $\,\ve=0$, as
$\,2\az'\nh=-\az^2\nh$.

Furthermore, $\,d\hs(\tz\gz+\az^{-\nnh1}Y\nh-\hs Q)=0$. 
In fact, $\,d\az=\az'\nh d\vt\,$ in (\ref{ape}), and similarly for
$\,\gz$, so that, from (\ref{ant}), 
$\,d\hs(\tz\gz)=\tz\hs d\hh\gz=\tz\gz'\nh d\vt=\tz\az^{-\nh2}d\vt$, 
and $\,\az\hs d\hs(\nh\az^{-\nnh1}Y\hn)
=dY\nh-\az^{-\nnh1}Y\hskip-2.3pt\az'\nh d\vt$. Also, 
$\,2(\tz-Y\hskip-2.3pt\az')=(Y\hskip-2.3pt\az+\hs\se\hn)\hh\az$ from
(\ref{app}) -- (\ref{cts}).
These relations yield 
$\,-\nh4\az\hs d\hs(\tz\gz+\az^{-\nnh1}Y\nh-\hs Q)
=4[(\nh\az\hs dQ-\hh dY\hn)-(\tz-Y\hskip-2.3pt\az')\hh\az^{-\nnh1}\nh d\vt]
=\nz\hh(\az\hs dQ-\hh dY\hn)-2(Y\hskip-2.3pt\az+\hs\se)\hs d\vt$, with 
$\,\nz=4$, which equals $\,0\,$ by (\ref{apd}.a).

Finally, (\ref{fcf}) and the second relation in (\ref{cts}) easily give 
(\ref{cns}.a). Thus, by (\ref{ven}) and (\ref{ape}),
$\,(\jq\az'\nh+F')\jq=(\jq\az'\nh+F')\hs d_\y\w\vt=\jq d_\y\w\az+d_\y\w F\,$
which, due to (\ref{cns}.a), equals $\,d_\y\w Y\nnh-\az\hs d_\y\w\jq$. 
At the same time, $\,-\imath_\y\w$ applied to (\ref{dqy}.i) yields
$\,d_\y\w Y\nnh-\az\hs d_\y\w\jq=-\nh2\jq\sz$. We thus get 
(\ref{cns}.b). To obtain (\ref{cns}.c), note that, from (\ref{ven}), 
$\,\Delta\az=Q\az''+\hs Y\hskip-2.3pt\az'$ which, by (\ref{app}) and (\ref{cns}.a),
equals $\,(Y\nnh-Q\az)\az'\nh=F\nh\az'\nh=-\nh F''\nh$, where the last
equality trivially follows from (\ref{fcf})
\end{proof}
%\begin{remark}\label{cnvrs}
Theorem~\ref{types} has a partial converse: if a nonconstant function
$\,\vt\,$
with real-hol\-o\-mor\-phic gradient $\,\y=\nabla\nh\vt$ on a K\"ah\-ler
surface $\,(M\nh,g)\,$ and a function $\,\az\,$ of the 
real variable $\,\vt\,$ satisfy (\ref{app}) and (\ref{cts}), then they must
also satisfy the Ric\-ci-\nh Hess\-i\-an equation (\ref{fma}) with $\,\sz\,$
given by (\ref{nec}) for $\,\nz=4$.

In fact, $\,b(\y,\,\cdot\,)=0$, where $\,b\,$
denotes the trace\-less Her\-mit\-i\-an symmetric $\,2$-ten\-sor field
$\,\az\nabla\nh d\vt+\rz-\sz\hn g$. Namely, (\ref{trv}) -- (\ref{imp}) and
(\ref{rwr}.b) yield
$\,4b(\y,\,\cdot\,)=2\az\,dQ-2\hs dY\nnh-4\sz\hs d\vt\,$ which, due to 
(\ref{nec}) and (\ref{ape}), equals
$\,2\az\,dQ-2\hs dY\nnh-(Y\hskip-2.3pt\az+\se)\hs d\vt$, and so
$\,-4\az\hs b(\y,\,\cdot\,)=
2\hs\az^2\nh d\hs(\tz\gz+\az^{-\nnh1}Y\nnh-Q)+
(\az\hs\se+4\ve Y\nnh-2\hh\tz)\hs d\vt$. (Note that, by (\ref{app}) and
(\ref{ant}), $\,4\ve=2\az'\nnh+\hs\az^2$ and 
$\,2\hs\az^2\nh d\hs(\tz\gz)=2\hh\tz\hs d\vt$.) Thus, $\,b=0$, since $\,b\,$ 
corresponds, via $\,g$, to a {\it com\-plex-lin\-e\-ar\/} bundle morphism 
$\,\tm\to\tm\nh$.
%\end{remark}

\section{The local K\"ah\-ler potentials}\label{lk}
\setcounter{equation}{0}
This and the following seven sections are devoted to proving
Theorem~\ref{reali}.

In an open set $\hs M\subseteq\bbR\nh^4\nh$ with the Cartesian coordinates
$\,x,x'\nh,u,u'$ arranged into the complex coordinates
$\,(x+ix'\nh,u+iu')\,$ for the complex plane $\,\bbC^2\nh=\bbR\nh^4$
carrying the standard complex structure $\,J$, one has $\,J^*\nnh dx=-\hh dx'$
and $\,J^*\nnh du=-\hh du'\nh$, so that, if a $\,C^\infty\nnh$ function
$\,f\,$ on $\,M\,$ only depends on $\,x\,$ and $\,u$, the relation
$\,2\hh i\hh\partial\hskip1.7pt\cro\nh f
=-\hh d\hskip1pt[\hn J^*\nnh d\hskip-.9ptf]$ yields, with subscripts
denoting partial differentiations, 
$\,2\hh i\hh\partial\hskip1.7pt\cro\nh f=
f\nh\nnh_{x\hn x}\w\hs dx\wedge dx'\nh
+f\nh\nnh_{x\hn u}\w(dx\wedge du'\nh+\hs du\wedge dx')
+f\nh\nnh_{u\hn u}\w\hs du\wedge du'\nh$, 
since $\,d\hskip-.9ptf=f\nh\nnh_x\w\hs dx+f\nh\nnh_u\w\hs du$. Furthermore, we
set
\begin{equation}\label{ved}
\y=\partial\hn_x\w\mathrm{\ \ and\ \ }\,\z=\partial\hn_u\w\quad\mathrm{(the\
real\ coordinate\ vector\ fields).}
\end{equation}
For the K\"ah\-ler metric $\,g\,$ on $\,M\,$ having the K\"ah\-ler form
$\,\kfo=2\hh i\hh\partial\hskip1.7pt\cro\hh\phi$, where the function
$\,\phi:M\to\bbR\,$ 
is assumed to depend on $\,x\,$ and $\,u\,$ only,
$\,2\phi\,$ is a K\"ah\-ler potential \cite[p.\,85]{besse} of $\,g$, and the
above formula for $\,2\hh i\hh\partial\hskip1.7pt\cro\nh f\nh$, with
$\,f=\phi$, becomes
\begin{equation}\label{two}
\kfo=\phi\nh_{x\hn x}\w\hs dx\wedge dx'\nh
+\phi\nh_{x\hn u}\w(dx\wedge du'\nh+\hs du\wedge dx')
+\phi\nh_{u\hn u}\w\hs du\wedge du'.
\end{equation}
Generally, for a skew-Her\-mit\-i\-an $\,2$-form
$\,\zeta=\jq\hs dx\wedge dx'\nh
+\js\hs(dx\wedge du'\nh+\hs du\wedge dx')
+\jp\hs du\wedge du'\,$ and the 
Her\-mit\-i\-an symmetric $\,2$-ten\-sor field $\,a\,$ with 
$\,\zeta=a(J\hs\cdot\,,\,\cdot\,)$ one has
$\,a=\jq(dx\otimes dx+\hs dx'\otimes dx')+\js\hs(dx\otimes du+\hs du\otimes dx
+dx'\otimes du'+\hs du'\otimes dx')+\jp(du\otimes du+du'\otimes du')$, due
to (\ref{wdg}), and so the components of $\,a\,$ relative to the coordinates
$\,(x,x'\nh,u,u')\,$ form the matrix
\begin{equation}\label{mtr}
\left[\begin{matrix}
\jq&0&\js&0\cr
0&\jq&0&\js\cr
\js&0&\jp&0\cr
0&\js&0&\jp\end{matrix}\right]\nnh\nh,\mathrm{\ \ with\ the\ determinant\
}\,(\jq\nh\jp\,-\,\js^2)^2\nh.
\end{equation}
When $\,a=g$, (\ref{two}) woth $\,\zeta=\kfo\,$ gives $\,(\jq,\js,\jp)
=(\phi\nh_{x\hn x}\w,\hs\phi\nh_{x\hn u}\w,\hs\phi\nh_{u\hn u}\w)$. Thus,
\begin{equation}\label{syl}
\phi\nh_{x\hn x}\w>0\,\mathrm{\ and\ }\,\jg>0\mathrm{,\ for\
}\,\jg\nh=\phi\nh_{x\hn x}\w\phi\nh_{u\hn u}\w\nh-\phi^{\hn2}_{\nh x\hn u},
\end{equation}
which amounts to Syl\-ve\-ster's criterion for positive definiteness of
$\,g$, namely, positivity of the upper left sub\-de\-ter\-mi\-nants of
(\ref{mtr}). From now on we set
\begin{equation}\label{set}
\begin{array}{l}
(\vt,\lambda,\jq,\js,\jp)=(\phi\nh_x\w,\phi\nh_u\w,\phi\nh_{x\hn x}\w,
\hs\phi\nh_{x\hn u}\w,\hs\phi\nh_{u\hn u}\w)\mathrm{,\hs\ so\hs\ that}\\
\jq\,>\,0\,\,\mathrm{\ and\ }\,\,\jg\,=\,\jq\nh\jp\,
-\,\js^2\hs>\,0\,\,\mathrm{\ due\ to\ (\ref{syl}).} 
\end{array}
\end{equation}
With $\,\mathrm{div},\hs\Delta\,$ denoting the $\,g$-di\-ver\-gence and
$\,g$-La\-plac\-i\-an, for $\,\vt,\lambda,\jq\,$ in (\ref{set}),
\begin{enumerate}
  \def\theenumi{{\rm\alph{enumi}}}
\item the functions $\,\vt,\hs\lambda\,$ have the hol\-o\-mor\-phic\
  $\,g$-gra\-di\-ents $\,\y=\partial\hn_x\w$ and $\,\z=\partial\hn_u\w$,
\item the other coordinate fields $\,J\y\,$ and $\,J\z\,$ are
hol\-o\-mor\-phic $\,g$-Kil\-ling fields,
\item $\jq=g(\y,\y)\,$ and $\,\Delta\hn\vt=\mathrm{div}\,\y
=[\hs\log\jg]_x\w$, while $\,\Delta\lambda
=\mathrm{div}\,\z=[\hs\log\jg]_u\w$.
\end{enumerate}
Namely, (\ref{ved}) and (\ref{mtr}) yield (a). Also, (b) 
follows since $\,\phi\,$ only depends on $\,x\,$ and $\,u$. Finally,
(\ref{mtr}) has the determinant $\,\jg^2\nh$, and so
$\,\jg\hs dx\wedge dx'\nh\wedge du\wedge du'$ is the volume form of $\,g$,
on which $\,\Lie\nh_\y\w,\Lie\nh_\z\w$ act -- see (\ref{ved}) -- via
partial differentiations $\,\partial\hn_x\w,\,\partial\hn_u\w$ of 
the $\,\jg\hs$ factor. Thus, 
$\,\mathrm{div}\,\y=[\hs\log\jg]_x\w$ and
$\,\mathrm{div}\,\z=[\hs\log\jg]_u\w$, cf.\
\cite[p.\,281]{kobayashi-nomizu}. On the other hand, by (a), (\ref{ved}) and
(\ref{set}), $\,g(\y,\y)=d_\y\w\vt=\partial\hn_x\w\vt
=\partial\hn_x\w\phi\nh_x\w=\phi\nh_{x\hn x}\w=\jq$.

For our $\,(\vt,\lambda)=(\phi\nh_x\w,\phi\nh_u\w)$, the mapping
$\,(x,u)\mapsto(\vt,\lambda)$ is locally dif\-feo\-mor\-phic due to
%positivity of $\,\jg\,$ in
(\ref{syl}), which makes 
$\,(\jq,\js,\jp)=(\phi\nh_{x\hn x}\w,\phi\nh_{x\hn u}\w,\phi\nh_{u\hn u}\w)$,
locally, a triple of real-val\-ued functions of the new variables
$\,\vt,\lambda$. Thus, from the chain rule, in matrix form
\begin{equation}\label{mtf}
\begin{array}{rl}
\mathrm{i)}&\left[\begin{matrix}\partial\nh_x\w\nnh
&\nnh\partial\nh_u\w\cr\end{matrix}\right]
=\left[\begin{matrix}\partial\hskip-1pt_\svt\w\nnh
&\nnh\partial\nnh_\lambda\w\cr
\end{matrix}\right]
\left[\begin{matrix}
\jq&\js\cr
\js&\jp\cr\end{matrix}\right]\nnh\nh,\quad\mathrm{ii)}
\left[\begin{matrix}d\vt\cr
d\lambda\cr\end{matrix}\right]
=\left[\begin{matrix}
\jq&\js\cr
\js&\jp\cr\end{matrix}\right]
{\left[\begin{matrix}dx\cr
du\cr\end{matrix}\right]}_{\phantom1}\hskip-5.2pt,\hskip7pt\mathrm{and\ so}\\
\mathrm{iii)}&(\jq\nh\jp\,-\,\js^2)
\left[\begin{matrix}\partial\hskip-1pt_\svt\w\nnh
&\nnh\partial\nnh_\lambda\w\cr\end{matrix}\right]
=\left[\begin{matrix}\partial\nh_x\w\nnh&\nnh\partial\nh_u\w\cr\end{matrix}
\right]
{\left[\begin{matrix}
\jp\hskip-3pt&-\nnh\js\cr
-\nnh\js\hskip-3pt&\jq\cr\end{matrix}\right]}_{\phantom{1_1}}\hskip-9.2pt,\\
\mathrm{iv)}&(\jq\nh\jp\,-\,\js^2)\left[\begin{matrix}dx\cr
du\cr\end{matrix}\right]
=\left[\begin{matrix}
\jp\hskip-3pt&-\nnh\js\cr
-\nnh\js\hskip-3pt&\jq\cr\end{matrix}\right]\left[\begin{matrix}d\vt\cr
d\lambda\cr\end{matrix}\right]\hskip-2.4pt.
\end{array}
\end{equation}
With subscripts still denoting partial differentiations, the obvious 
integrability conditions
$\,\jq\nh_u\w\nh-\js\nh_x\w=\js\nh_u\w\nh-\jp\hskip-1.3pt_x\w=0\,$ and
(\ref{set}) give, due to (\ref{mtf}.i),
\begin{equation}\label{eqv}
\js\jq\hn_\svt\w+\jp\jq\nh_\lambda\w=\jq\hn\js\nh_\svt\w
+\js\nh\js\nh_\lambda\w,\hskip7pt
\js\nh\js\nh_\svt\w+\jp\nh\js\nnh_\lambda\w=\jq\nh\jp\hskip-1.3pt_\svt\w
+\js\nh\jp\hskip-1.3pt_\lambda\w,\hskip7pt
\jq>0\hh,\hskip7pt\jq\nh\jp>\js^2\nnh.
\end{equation}
Conversely, if functions $\,\jq,\js,\jp\,$ of the variables $\,\vt,\lambda\,$
satisfy (\ref{eqv}), then, locally,
\begin{enumerate}
\item[{\rm(d)}]$(\jq,\js,\jp)
=(\phi\nh_{x\hn x}\w,\phi\nh_{x\hn u}\w,\phi\nh_{u\hn u}\w)\,$ for a function
$\,\phi$, with (\ref{syl}), of the variables $\,x,u\,$ related to
$\,\vt,\lambda\,$ via $\,(\vt,\lambda)=(\phi\nh_x\w,\phi\nh_u\w)$, and
$\,\jq,\js,\jp\,$ determine each of $\,\phi,x,u\,$ uniquely up to additive
constants.
\end{enumerate}
In fact, the equalities in (\ref{eqv}) state precisely that the vector fields
$\,\jq\partial\hskip-1pt_\svt\w+\js\hn\partial\nnh_\lambda\w$ and
$\,\js\hn\partial\hskip-1pt_\svt\w+\jp\nh\partial\nnh_\lambda\w$ commute or,
equivalently, the $\,1$-forms
$\,(\jq\nh\jp-\js^2)^{-\nnh1}\nh(\jp\hs d\vt-\js\hs d\lambda)\,$ and
$\,(\jq\nh\jp-\js^2)^{-\nnh1}\nh(\jq\,d\lambda-\js\hs d\vt)$, dual to them,
are closed, and we may declare these vector fields (or, $\,1$-forms) to be
$\,\partial\nh_x\w,\,\partial\nh_u\w$ or, respectively, $\,dx,\,du$. Now
that, locally, $\,x,u\,$ are defined, up to additive constants, we obtain
$\,\phi\,$ by solving the system $\,(\phi\nh_x\w,\phi\nh_u\w)=(\vt,\lambda)$,
where $\,\vt,\lambda\,$ are treated as functions of $\,x,u\,$ via the
resulting locally dif\-feo\-mor\-phic coordinate change
$\,(\vt,\lambda)\mapsto(x,u)$. Closedness of $\,\vt\,dx+\lambda\,du$
and the equality $\,(\jq,\js,\jp)
=(\phi\nh_{x\hn x}\w,\phi\nh_{x\hn u}\w,\phi\nh_{u\hn u}\w)\,$ are obvious:
our choice of $\,dx\,$ and $\,\,du$ gives 
(\ref{mtf}.iv), and hence (\ref{mtf}.ii), so that
$\,d\vt\wedge dx+d\lambda\wedge du=0$.

The $\,g$-La\-plac\-i\-ans of $\,\vt\,$ and $\,\lambda\,$ can also be
expressed as
\begin{enumerate}
\item[{\rm(e)}]$\Delta\hn\vt\,=\,\jq\hn_\svt\w\,+\,\js\nh_\lambda\w\,\,$ and 
$\,\,\Delta\lambda\,=\,\js\nh_\svt\w\,+\,\jp\hskip-1.3pt_\lambda\w$, while
\item[{\rm(f)}]$\jg\nnh_x\w
=(\jq\hn_\svt\w+\js\nh_\lambda\w)\jg\,$ and 
$\,\jg\nnh_u\w
=(\js\nh_\svt\w+\jp\hskip-1.3pt_\lambda\w)\jg$, for
$\,\jg\nh=\jq\nh\jp-\js^2\nh$.
\end{enumerate}
To see this, first note that, by (\ref{mtf}.i), %one gets 
$\,(\jq\nh\jp-\js^2)\nh_x\w=\jq(\jq\nh\jp-\js^2)\nh_\svt\w+\js(\jq\nh\jp
-\js^2)\nh_\lambda\w=
\jq(\jq\nh\jp\hskip-1.3pt_\svt\w+\js\nh\jp\hskip-1.3pt_\lambda\w
-\js\nh\js\nh_\svt\w)
-\js(\jq\hn\js\nh_\svt\w+\js\nh\js\nh_\lambda\w
-\jp\jq\nh_\lambda\w)-\js^2\nh\js\nh_\lambda\w
+\jq\nh\jp\jq\hn_\svt\w$. Using (\ref{eqv}) to replace the two three-term sums
in parentheses by $\,\jp\nh\js\nh_\lambda\w$ and $\,\js\jq\hn_\svt\w$,
we thus obtain the first part of (f). For the second one we similarly rewrite 
$\,(\jq\nh\jp-\js^2)\nh_u\w
=\js(\jq\nh\jp-\js^2)\nh_\svt\w+\jp(\jq\nh\jp-\js^2)\nh_\lambda\w$ as 
$\,\js(\jq\nh\jp\hskip-1.3pt_\svt\w-\jp\js\nh_\lambda\w-\js\nh\js\nh_\svt\w)
-\js^2\nh\js\nh_\svt\w+\jp(\jp\jq\nh_\lambda\w+\js\jq\hn_\svt\w
-\js\nh\js\nh_\lambda\w)+\jq\nh\jp\jp\hskip-1.3pt_\lambda\w$, and use 
analogous three-term replacements based on (\ref{eqv}). Now (e) follows 
from (c), (\ref{set}) and (f).
\begin{theorem}\label{equiv}
In\/ $\,\bbC^2$ with the complex coordinates\/
$\,(x+ix'\nh,u+iu')$,  
given an open subset\/ $\,M\,$ and a function\/ $\,\phi:M\to\bbR\,$ of the real
variables\/ $\,x,u$, having the property\/ {\rm(\ref{syl})}, let\/ $\,g\,$ be
the K\"ah\-ler metric on\/ $\,M\,$ with the K\"ah\-ler potential\/ $\,2\phi$.
The following two conditions are equivalent.
\begin{enumerate}
  \def\theenumi{{\rm\roman{enumi}}}
\item[(i)]The special Ric\-ci-\nh Hess\-i\-an equation\/ {\rm(\ref{fma})} -- 
{\rm(\ref{anz})} holds on\/ $\,M\,$ for\/ $\,\vt=\phi\nh_x\w$,  
and\/ $\,d\jq\wedge d\vt\ne0\,$ everywhere in\/ $\,M\nh$, with\/
$\,\jq=g(\nabla\nh\vt,\nnh\nabla\nh\vt)$. Thus, by Theorem\/~{\rm\ref{types}},
one has\/ {\rm(\ref{app})} and\/ {\rm(\ref{cns}.a)}, where\/
$\,Y\nh=\Delta\hn\vt\,$ and\/ $\,F\,$ is the function of\/ $\,\vt\,$
characterized by\/ {\rm(\ref{fcf})}, for the constants\/ $\,\tz,\kappa\,$
in\/ {\rm(\ref{cts})}.
\item[(ii)]The triple\/
$\,(\jq,\js,\jp)=(\phi\nh_{x\hn x}\w,\phi\nh_{x\hn u}\w,\phi\nh_{u\hn u}\w)\,$ 
of functions of the new variables\/ $\,(\vt,\lambda)
=(\phi\nh_x\w,\phi\nh_u\w)\,$ satisfies\/ {\rm(\ref{eqv})} with\/
$\,\jq\nh_\lambda\w\ne0\,$ everywhere, as well as the equations\/ 
$\,\js\nh_\svt\w+\jp\hskip-1.3pt_\lambda\w\nh=\js\nh\az+\jr$, 
$\,\,\jq\hn_\svt\w+\js\nnh_\lambda\w\nh=\jq\az+F\nh$, 
$\,\,\jr\nh_\svt\w\nh=-\js\nh\az'\nh$, $\,\jr\nnh_\lambda\w\nh
=\jq\az'\nh+F'$ for some function\/ $\,\jr\,$ and\/
$\,(\hskip2.5pt)'\nh=d/d\vt$.
\end{enumerate}
\end{theorem}
\begin{proof}By (\ref{cns}.b), equation (\ref{fma}) is, in case (i),
equivalent to
\begin{equation}\label{eqt}
2\az\nabla\nh d\vt\,+\,2\rz\,=\,-(\jq\az'\nh+F')\hh g\hh,
\end{equation}
where all the terms are Her\-mit\-i\-an symmetric $\,2$-ten\-sor fields, and
hence correspond, via $\,g$, to {\it com\-plex-lin\-e\-ar\/} bundle morphisms
$\,\tm\to\tm\nh$. Thus, (\ref{eqt}) amounts to
\begin{enumerate}
\item[{\rm(f)}]equalities of the images of both sides in (\ref{eqt})
under $\,\imath_\y\w$ and $\,\imath_\z\w$.
\end{enumerate}
The equality of the $\,\imath_\y\w$-im\-ages is, by (\ref{rwr}.b) and
(\ref{trv}), the
result of applying $\,d$, via (\ref{ape}), to the relation (\ref{cns}.a) in
(i): $\,Y\nnh-Q\az=F\nh$. This last
relation and (e), with $\,Y\nh=\Delta\hn\vt\,$ due to (\ref{ven}), 
show that (i) implies the equality $\,\jq\hn_\svt\w\hn+\js\nnh_\lambda\w\nh
=\jq\az+F\,$ in (ii). Defining $\,\jr\,$ to be
$\,\js\nh_\svt\w+\jp\hskip-1.3pt_\lambda\w-\js\nh\az\,$ 
we get $\,\js\nh_\svt\w+\jp\hskip-1.3pt_\lambda\w=\js\nh\az+\jr$. On the
other hand, the equality
of the $\,\imath_\z\w$-im\-ages in (\ref{eqt}) reads
\begin{equation}\label{img}
\az\,d\js\,-\,d\hs\Delta\lambda\,=\,-(\jq\az'\nh+F')\,d\lambda\hh.
\end{equation}
In fact, the first term equals $\,\az\,d\js\,$ since, for the two commuting
gradients $\,\y=\nabla\nh\vt$ and $\,\z=\nabla\nh\lambda$, one has
$\,2\nabla\!_\z\w d\vt=\hs d\hh[g(\y,\z)]\,$ or, in local coordinates, 
$\,2\z\hh^k\y_{,\hh jk}=\z\hh^k\y_{,\hh jk}+\y\hh^k\z_{,\hh jk}
=(\y\hh^k\z_k\w)\hn_{,\hh j}\w$, and $\,\js=\phi\nh_{x\hn u}\w
=g(\y,\z)\,$ by (\ref{mtr}) and (\ref{ved}). The second term is
$\,-d\hs\Delta\lambda\,$ due to (a) and (\ref{rhg}). By (e),
$\,\jr=\js\nh_\svt\w+\jp\hskip-1.3pt_\lambda\w-\js\nh\az=\Delta\lambda
-\js\nh\az$, and so (\ref{img}) becomes
$\,\az\,d\js-d\hh(\jr+\js\nh\az)=-(\jq\az'\nh+F')\,d\lambda$, that is,
according to (\ref{ape}),
\[
d\jr\,=\,(\jq\az'\nh+F')\,d\lambda\,-\,\js\hs d\az\,
=\,(\jq\az'\nh+F')\,d\lambda\,-\,\js\nh\az'\nh d\vt
\]
or, in other words,
$\,\jr\nh_\svt\w=-\js\nh\az'$ and $\,\jr\nnh_\lambda\w=\jq\az'\nh+F'\nh$.
Consequently, (i) implies (ii), since (\ref{eqv}) arises as
the integrability conditions
$\,\jq\nh_u\w\nh-\js\nh_x\w=\js\nh_u\w\nh-\jp\hskip-1.3pt_x\w=0\,$ 
combined with (\ref{syl}), and the equality
$\,d\jq=\jq\hn_\svt\w\hs d\vt+\jq\nh_\lambda\w\hs d\lambda\,$ yields 
$\,d\jq\wedge d\vt=-\jq\nh_\lambda\w\hs d\vt\wedge d\lambda$.

Conversely, assuming (ii), we get (i) from (f). Namely, as we saw above,
the equality of the $\,\imath_\y\w$-im\-ages in (\ref{eqt}) arises by applying
$\,d\,$ to $\,Y\nnh-Q\az=F\nh$, that is -- cf.\ (e) -- to 
$\,\jq\hn_\svt\w\hn+\js\nnh_\lambda\w\nh=\jq\az+F\nh$. Also, (\ref{img})
follows from (ii) and (e):
$\,\az\,d\js-d\hs\Delta\lambda
=\az\,d\js-d\hh(\js\nh_\svt\w+\jp\hskip-1.3pt_\lambda\w)
=\az\,d\js-d\hh(\js\nh\az+\jr)
=-d\jr-\js\hs d\az
=-d\jr-\js\nh\az'\nh d\vt
=\jr\nh_\svt\w\hs d\vt-d\jr
=-\jr\nnh_\lambda\w\hs d\lambda
=-(\jq\az'\nh+F')\,d\lambda$. This completes the proof.
\end{proof}
Note that (e), (f), Theorem~\ref{equiv}(ii) and (\ref{ved}) give
\[
\Delta\hn\vt=\jq\az+F\nh,\quad\Delta\lambda=\js\nh\az+\jr,\quad
d_\y\w\jg\nh=\jg\Delta\hn\vt,\quad
d_\z\w\jg\nh=\jg\Delta\lambda.
\]

\section{Some linear algebra}\label{sl}
\setcounter{equation}{0}
We now proceed to discuss the first-or\-der system equivalent, as we saw in
the last section (Theorem~\ref{equiv}), to the K\"ah\-ler-po\-ten\-tial
problem, the solution 
of which amounts to proving Theorem~\ref{reali}. The main result established
here, Theorem~\ref{affsb}, will lead -- in Section~\ref{ue} -- to a
unique\hs-extension property of integral lines, which results in applicability 
of the Car\-tan-K\"ah\-ler theorem to our situation.

Theorem~\ref{equiv} reduces constructing local examples of special
Ric\-ci-\nh Hess\-i\-an equations (\ref{fma}) -- (\ref{anz}) with
$\,d\jq\wedge d\vt\ne0$, on K\"ah\-ler surfaces, which is a fourth-or\-der
problem in the K\"ah\-ler potential $\,2\phi$, to the following system of 
qua\-si-lin\-e\-ar first-or\-der partial differential equations, with
subscripts representing partial derivatives:
\begin{equation}\label{foe}
\begin{array}{rl}
\mathrm{a)}&\jq\hn_\svt\w+\js\nnh_\lambda\w-\jq\az-F=0\hh,\\
\mathrm{b)}&\js\nh_\svt\w+\jp\hskip-1.3pt_\lambda\w-\js\nh\az-\jr=0\hh,\\
\mathrm{c)}&\jq\nh\jp\hskip-1.3pt_\svt\w+\js\nh\jp\hskip-1.3pt_\lambda\w
-\js\nh\js\nh_\svt\w-\jp\nh\js\nh_\lambda\w=0\hh,\\
\mathrm{d)}&\js\jq\hn_\svt\w+\jp\jq\nh_\lambda\w-\jq\hn\js\nh_\svt\w
-\js\nh\js\nnh_\lambda\w=0\hh,\\
\mathrm{e)}&\jr\nh_\svt\w+\js\nh\az'\nh=0\hh,\\
\mathrm{f)}&\jr\nnh_\lambda\w-\jq\az'\nh-F'\nh=0\hh,
\end{array}
\end{equation}
on which one imposes the additional conditions 
\begin{equation}\label{det}
\jq\nh\jp>\js^2,\qquad\jq\,>\,0\hh,\qquad\jq\nh_\lambda\w\hs\ne\,0\,\mathrm{\ 
\ \ everywhere.}
\end{equation}
Generally, if $\,\jq,\js,\jp\,$ are real-val\-ued functions 
of the real variables $\,\vt,\lambda\,$ and  subscripts denote partial
differentiations, writing 
$\,d[(\jq\nh\jp\nnh-\nnh\js^2)^{-\nnh1}\nh(\jp\hs d\vt\nh
-\nh\js\hs d\lambda)]=\varPhi\,d\vt\wedge d\lambda\,$ 
and $\,d[(\jq\nh\jp\nnh-\nnh\js^2)^{-\nnh1}\nh(\js\hs d\vt\nh
  -\nh\jq\,d\lambda)]=\varPsi\,d\vt\wedge d\lambda\,$ at points where
$\,\jq\nh\jp\ne\js^2\nh$, one easily verifies that
\begin{equation}\label{php}
\left[\begin{matrix}
\varPhi\cr
\varPsi\end{matrix}\right]=
\left[\begin{matrix}
\js&\jp\cr
\jq&\js\end{matrix}\right]
\left[\begin{matrix}
\jq\nh\jp\hskip-1.3pt_\svt\w+\js\nh\jp\hskip-1.3pt_\lambda\w
-\js\nh\js\nh_\svt\w-\jp\nh\js\nh_\lambda\w\cr
\js\jq\hn_\svt\w+\jp\jq\nh_\lambda\w-\jq\hn\js\nh_\svt\w
-\js\nh\js\nnh_\lambda\w\end{matrix}\right]\nnh.
\end{equation}
For the remainder of this section we treat the letter symbols in (\ref{foe}) --
(\ref{det}) as real variables, so as to derive some consequences of conditions 
(\ref{foe}) -- (\ref{det}) just by using linear algebra.
\begin{lemma}\label{cnseq}If\/
$\,(\jq,\js,\jp\nh,\jr,\jq\hn_\svt\w,\js\nh_\svt\w,\jp\hskip-1.3pt_\svt\w,
\jr\nh_\svt\w,\jq\nh_\lambda\w,\js\nh_\lambda\w,
\jp\hskip-1.3pt_\lambda\w,\jr\nh_\lambda\w)\in\bbR\nh^{12}$ satisfies 
the conditions\/ {\rm(\ref{foe}.a)} -- {\rm(\ref{foe}.d)}, with some\/
$\,(\az,F)\in\rto\nh$, then
\begin{equation}\label{qbt}
\begin{array}{rl}
\mathrm{i)}&\jq\nh\jp\hskip-1.3pt_\svt\w+\jp\jq\hn_\svt\w
-2\js\nh\js\nh_\svt\w-(\jq\nh\jp-\js^2)\az-\jp F+\js\jr=0\hh,\\
\mathrm{ii)}&\jq\nh\jp\hskip-1.3pt_\lambda\w+\jp\jq\nh_\lambda\w
-2\js\nh\js\nnh_\lambda\w-\jq\jr+\js F=0\hh.
\end{array}
\end{equation}
\end{lemma}
\begin{proof}The left-hand side of (\ref{qbt}.i) is, obviously, 
$\,(\jq\nh\jp\hskip-1.3pt_\svt\w+\js\nh\jp\hskip-1.3pt_\lambda\w
-\js\nh\js\nh_\svt\w-\jp\nh\js\nh_\lambda\w)
+(\jq\hn_\svt\w+\js\nnh_\lambda\w-\jq\az-F)\jp
-(\js\nh_\svt\w+\jp\hskip-1.3pt_\lambda\w-\js\nh\az-\jr)\js$, and each sum in
parentheses vanishes due to (\ref{foe}). Similarly,
(\ref{qbt}.ii) has the left-hand side
$\,(\js\jq\hn_\svt\w+\jp\jq\nh_\lambda\w-\jq\hn\js\nh_\svt\w
-\js\nh\js\nnh_\lambda\w)
+(\js\nh_\svt\w+\jp\hskip-1.3pt_\lambda\w-\js\nh\az-\jr)\jq
-(\jq\hn_\svt\w+\js\nnh_\lambda\w-\jq\az-F)\js=0+0+0$.
\end{proof}
With subscripts denoting partial differentiations, (\ref{foe}.e\hs-\hh f)
and (\ref{qbt}) read
\begin{equation}\label{dqb}
\begin{array}{rl}
\mathrm{i)}&d\jr\,=\,-\js\nh\az'd\vt\,
+\,(\jq\az'\nh+\hs F')\,d\lambda\hh,\\
\mathrm{ii)}&d\jg=\hs(\jg\nh\az
+\jp F-\js\jr)\,d\vt\,
+\,\nh(\jq\jr-\js F)\,d\lambda\,\mathrm{\ for\ }\,\jg=\jq\nh\jp-\js^2\nh.
\end{array}
\end{equation}
Since $\,d\jg\nh
=\jg\hh[(\jq\hn_\svt\w+\js\nh_\lambda\w)\,dx
+(\js\nh_\svt\w+\jp\hskip-1.3pt_\lambda\w)\,du]\,$ due to (f), 
%Rather than from (\ref{qbt}),
one can also obtain (\ref{dqb}.ii) from 
(\ref{mtf}.iv), with $\,\jg=\jq\nh\jp-\js^2\nh$, and (\ref{foe}.a\hs-b).
\begin{lemma}\label{onlyz}Let\/
$\,(\dot\vt,\dot\lambda,\jq,\js,\jp\nh,\jr)\in\bbR\nh^6$ have\/ 
$\,(\dot\vt,\dot\lambda)\ne(0,0)\,$ and\/ $\,\jq\nh\jp>\js^2\nh$.
\begin{enumerate}
\item[{\rm(a)}]$\jq\nh\jp>0\,$ and\/ $\,\varPsi\ne0$, where\/ 
$\,\varPsi=\jp\dot\vt^2\nh-2\js\dot\vt\dot\lambda+\jq\dot\lambda^2\nh$.
\item[{\rm(b)}]$(0,0,0,0,0,0)\,$ is the only vector\/
$\,(\jq\hn_\svt\w,\js\nh_\svt\w,\jp\hskip-1.3pt_\svt\w,
\jq\nh_\lambda\w,\js\nh_\lambda\w,\jp\hskip-1.3pt_\lambda\w)
\in\bbR\nh^6$ with
\[
\left[\begin{matrix}
1&0&0&0&1&0\cr
0&1&0&0&0&1\cr
0&-\nh\js&\jq&0&-\nh\jp&\js\cr
\js&-\jq&0&\jp&-\nh\js&0\cr
\dot\vt&0&0&\dot\lambda&0&0\cr
0&\dot\vt&0&0&\dot\lambda&0\cr
0&0&\dot\vt&0&0&\dot\lambda\end{matrix}\right]
\left[\begin{matrix}
\jq\hn_\svt\w\cr
\js\nh_\svt\w\cr
\jp\hskip-1.3pt_\svt\w\cr
\jq\nh_\lambda\w\cr
\js\nh_\lambda\w\cr
\jp\hskip-1.3pt_\lambda\w\end{matrix}\right]
=\left[\begin{matrix}
0\cr0\cr0\cr0\cr0\cr0\end{matrix}\right]\nnh.
\]  
\item[{\rm(c)}]The first four rows of the above\/ 
$\,7\times6\,$ matrix are linearly independent.
\end{enumerate}
\end{lemma}
\begin{proof}Assertion (a) follows as $\,\varPsi\,$ is positive or negative
definite in $\,(\dot\vt,\dot\lambda)$. Next, one easily verifies that, for
$\,\varPsi\,$ as in (a),
%the combination
\[%\begin{equation}\label{cmb}
(\jp\dot\vt-\js\dot\lambda)\dot\lambda\text{\smbf r}\hn_1\w
+\varPsi\text{\smbf r}\hn_2\w
+\dot\lambda^2\nh\text{\smbf r}\hn_4\w
-\jp\dot\lambda\text{\smbf r}\hn_5\w
+(2\js\dot\lambda-\jp\dot\vt)\text{\smbf r}\hn_6\w
=(0,0,0,0,0,\varPsi)\hh.
\]%\end{equation}
Depending on whether $\,\dot\lambda\ne0\,$
(or, $\,\dot\lambda=0\,$ and hence $\,\dot\vt\ne0$), the $\,6\times6\,$ matrix
with the rows
$\,\text{\smbf r}\hn_1\w,\,\text{\smbf r}\hn_2\w,\,\text{\smbf r}\hn_3\w
+\js\text{\smbf r}\hn_2\w,\,\text{\smbf r}\hn_4\w
-\js\text{\smbf r}\hn_1\w+\jq\text{\smbf r}\hn_2\w$ 
followed by $\,\text{\smbf r}\hn_6\w-\dot\vt\text{\smbf r}\hn_2\w$
(or, respectively, by
$\,\text{\smbf r}\hn_5\w-\dot\vt\text{\smbf r}\hn_1\w$), and then by
$\,(0,0,0,0,0,\varPsi)\,$ displayed above, is upper triangular, with the diagonal entries
$\,1,1,\jq,\jp,\dot\lambda,\varPsi$ or, respectively,
$\,1,1,\jq,\jp,\nnh-\nh\dot\vt,\varPsi$, all of them nonzero by 
(a). These six new rows thus form a basis of $\,\bbR\nh^6$ (so that the matrix
has rank $\,6$), while the first four 
of them are linear combinations of the first four original rows, which proves
both (b) and (c).
\end{proof}
\begin{theorem}\label{affsb}Given vectors\/
$\,(\dot\vt,\dot\lambda,\jq,\js,\jp\nh,\jr)\in\bbR\nh^6$ and\/ 
$\,(\az,F\nh,\az'\nh,F')\in\bbR\nh^4$ with\/
$\,(\dot\vt,\dot\lambda)\ne(0,0)\,$ and\/ $\,\jq\nh\jp>\js^2\nh$, the set of
all
\begin{equation}\label{qts}
(\jq\hn_\svt\w,\js\nh_\svt\w,\jp\hskip-1.3pt_\svt\w,\jr\nh_\svt\w,
\jq\nh_\lambda\w,\js\nh_\lambda\w,\jp\hskip-1.3pt_\lambda\w,\jr\nh_\lambda\w)
\,\in\,\bbR\nh^8
\end{equation}
satisfying\/ {\rm(\ref{foe})} forms a two\hs-di\-men\-sion\-al af\-fine 
sub\-space\/ 
$\,\mathcal{A}\,$ of\/ $\,\bbR\nh^8\nh$. Furthermore, the restriction to\/
$\,\mathcal{A}\,$ of the linear operator\/
$\,\varPhi:\bbR\nh^8\nh\to\bbR\nh^4$ sending\/ {\rm(\ref{qts})} to
\begin{equation}\label{dqd}
(\dot\jq,\dot\js,\dot\jp\nh,\dot\jr)
=(\jq\hn_\svt\w\dot\vt+\jq\nh_\lambda\w\dot\lambda,\,
\js\nh_\svt\w\dot\vt+\js\nh_\lambda\w\dot\lambda,\,
\jp\hskip-1.3pt_\svt\w\dot\vt+\jp\hskip-1.3pt_\lambda\w\dot\lambda,\,
\jr\nh_\svt\w\dot\vt+\jr\nh_\lambda\w\dot\lambda)
\end{equation}
is an af\-fine iso\-mor\-phism\/ 
$\,\varPhi:\mathcal{A}\to\mathcal{L}\,$ onto the two\hs-di\-men\-sion\-al
af\-fine sub\-space\/ $\,\mathcal{L}$  of\/ $\,\bbR\nh^4$ consisting of all\/ 
$\,(\dot\jq,\dot\js,\dot\jp\nh,\dot\jr)\,$ such that
\begin{equation}\label{dge}
\begin{array}{l}
\dot\jr\,=\,-\js\nh\az'\nh\dot\vt\,+\,(\jq\az'\nh+F')\dot\lambda\hh,\\
\jq\nh\dot\jp+\jp\dot\jq-2\js\nh\dot\js
=[(\jq\nh\jp-\js^2)\az+\jp F-\js\jr]\dot\vt+(\jq\jr-\js F)\dot\lambda\hh.
\end{array}
\end{equation}
\end{theorem}
\begin{proof}That $\,\mathcal{A}\subseteq\bbR\nh^8\nh$, or 
$\,\mathcal{L}\subseteq\bbR\nh^4\nh$, if nonempty, is an af\-fine sub\-space,
clearly follows since (\ref{foe}), or (\ref{dge}), is a system of
nonhomogeneous linear equations imposed on (\ref{qts}) or, respectively, 
$\,(\dot\jq,\dot\js,\dot\jp\nh,\dot\jr)$. Also, $\,\mathcal{A}\,$ is 
nonempty, and two\hs-di\-men\-sion\-al, being the pre\-im\-age of 
$\,(\jq\az+F\nh,\js\nh\az+\jr,0,0,-\js\nh\az'\nh,\jq\az'\nh+F')\,$ under the
obvious linear operator $\,\bbR\nh^8\nh\to\bbR\nh^6\nh$. Namely, this operator
is surjective: it equals the direct sum of the identity operator
$\,\rto\nh\to\rto\nh$, acting on $\,(\jr\nh_\svt\w,\jr\nh_\lambda\w)$, and
an operator $\,\bbR\nh^6\nh\to\bbR\nh^4$ which has rank $\,4\,$ due to
Lemma~\ref{onlyz}(c). Next, $\,\varPhi\,$ maps $\,\mathcal{A}\,$ into
$\,\mathcal{L}$, which one
sees adding (\ref{foe}.e), or (\ref{qbt}.i), times $\,\dot\vt\,$ to 
(\ref{foe}.f), or (\ref{qbt}.ii), times $\,\dot\lambda$, and then using 
(\ref{dqd}). Thus, $\,\mathcal{L}\,$ is nonempty, and 
$\,\dim\mathcal{L}=2$, the matrix of the homogeneous system associated
with (\ref{dge}) having rank $\,2\,$ since Lemma~\ref{onlyz}(a) %(\ref{qbn})
gives $\,\jq\ne0$.
Let $\,\mathcal{A}\hn'\nh\subseteq\bbR\nh^8$ and 
$\,\mathcal{L}\hn'\nh\subseteq\bbR\nh^4$ now be the vector sub\-spaces
parallel to $\,\mathcal{A}\, $ and $\,\mathcal{L}$.
Our assertion will thus follow once we establish injectivity of 
$\,\varPhi:\mathcal{A}\to\mathcal{L}$, that is, injectivity of its linear
part $\,\varPhi:\mathcal{A}\hn'\nh\to\mathcal{L}\hn'\nh$. Equivalently, we
need to show that zero is the only vector (\ref{qts}) lying in
$\,\mathcal{A}\hn'$ and having the
$\,\varPhi$-im\-age $\,(0,0,0,0)\,$ or, in other words, the only solution to
the matrix equation in Lemma~\ref{onlyz},  
with $\,(\dot\jr,\jr\nh_\svt\w,\jr\nh_\lambda\w)=(0,0,0)$.
This, however, is precisely what Lemma~\ref{onlyz}(b) states.
%, thus completing the proof.
\end{proof}

\section{The existence of solutions}\label{es}
\setcounter{equation}{0}
After the preceding foray into linear algebra, we now return to treating
(\ref{foe}) as a system of qua\-si-lin\-e\-ar first-or\-der
partial differential equations with four unknown real-val\-ued functions
$\,\jq,\js,\jp\nh,\jr\,$ of the real variables $\,\vt,\lambda$, subject to
the additional conditions (\ref{det}).

Subscripts again denote partial differentiations, while $\,\az\,$ and
$\,F\,$ are functions of the variable $\,\vt$, also depending on three fixed
real constants $\,\ve,\tz,\kappa$, so that
\begin{equation}\label{tap}
\begin{array}{l}
2\az'\nh+\hs\az^2=4\ve\mathrm{,\nnh\ where\ }\,(\,\,)'\nh=d/d\vt\hh,\\
4\ve F\nh=\hs\tz(2\hn-\hn\vt\az)+4\ve\kappa\az\,\mathrm{\ if\ }\hs\ve\ne0\hh,\\
F=\,\kappa\az\hh-\hh2\hh\tz/(3\az^2)\,\mathrm{\ when\ }\,\ve\hs=\hs0\hh,
\end{array}
\end{equation}
In addition, it is natural to assume here that
\begin{equation}\label{rng}
\vt\,\mathrm{\ ranges\ over\ the\ domain\ of\ }\,\az\,\mathrm{\ (which\ is\
also\ the\ domain\ of\ }\,F)\hh.
\end{equation}
Consequently, $\,\az\,$ and $\,F\,$ satisfy the ordinary differential equations
\begin{equation}\label{adp}
\az''\hh+\,\az\az'\hh=\,0\hh,\qquad F''\hh=\,-F\nh\az'.
\end{equation}
\begin{theorem}\label{exist}For any fixed\/ $\,\az,F\hs$ as in\/
{\rm(\ref{tap})} -- {\rm(\ref{rng})}, real-an\-a\-lyt\-ic solutions\/ 
$\,\zb=(\jq,\js,\jp\nh,\jr)\,$ to\/ {\rm(\ref{foe})} -- {\rm(\ref{det})} 
exist, locally, on a neighborhood of any\/ 
$\,(\vt,\lambda)\in\rto$ with the property\/ {\rm(\ref{rng})}.

More precisely, one obtains a lo\-cal\-ly-u\-nique such solution\/ $\,\zb\,$
by prescribing\/ $\,\zb$ and the partial derivatives\/ 
$\,\zb_\svt\w,\zb\nh_\lambda\w$ real-an\-a\-lyt\-ic\-al\-ly along an arbitrary
real-an\-a\-lyt\-ic embedded curve\/ $\,t\mapsto(\vt,\lambda)\in\rto\nh$,
%of an interval,
so as to satisfy\/ {\rm(\ref{foe})}, {\rm(\ref{det})},
{\rm(\ref{rng})}, and the condition\/ $\,\dzb=\dot\vt\zb\hn_\svt\w
+\dot\lambda\hs\zb\nh_\lambda\w$, where\/ $\,(\,\,)\dot{\,}=\,d/dt$.
\end{theorem}
We prove Theorem~\ref{exist} at the end of Section~\ref{ei}. 
As we point out in Remark~\ref{chdat}, there is an
in\-fi\-nite\hs-di\-men\-sion\-al freedom of choosing the data described in
the second paragraph of Theorem~\ref{exist}

\section{The associated exterior differential system}\label{ae}
\setcounter{equation}{0}
By an {\it exterior differential system\/} on a manifold $\,M\,$ one means an
ideal $\,\mathcal{I}\hs$ in the graded algebra $\,\varOmega\hh^*\nnh M\nh$, closed
under exterior differentiation; its {\it integral manifolds\/} (or, {\it 
integral elements\/}) are those sub\-man\-i\-folds of
$\,M\,$ (or, sub\-spaces of its tangent spaces) on which every form in
$\,\mathcal{I}\hs$ 
vanishes \cite[pp.\,16,\,65]{bryant-chern-gardner-goldschmidt-griffiths}.
When such %sub\-man\-i\-folds/sub\-spaces are of
objects have dimension $\,1\,$ or $\,2$, we 
call them {\it integral curves}/\nnh{\it sur\-faces\/} or {\it lines}/\nnh{\it
planes}.

If $\,E\subseteq\tzm\,$ is a $\,p\hs$-di\-men\-sion\-al integral element of
$\,\mathcal{I}\nh$, one sets  
\cite[pp.\,67-68]{bryant-chern-gardner-goldschmidt-griffiths}:
\begin{equation}\label{hee}
\begin{array}{l}
H\nh(E)=\{v\in\tzm\nnh:\zeta(v,e_1\w,\dots,e_p\w)\,=\,0\,\mathrm{\ for\ all\
}\hs\zeta\in\mathcal{I}\cap\varOmega\hh^{p\hs+1}\nnh M\}\\
\mathrm{and\ }\,r(E)\hs=\hs\dim H\nh(E)-(p+1)\mathrm{,\ for\ any\ basis\
}\,e_1\w,\dots,e_p\w\mathrm{\ of\ }\,E,
\end{array}
\end{equation}
so that $\,H\nh(E)\,$ is a vector sub\-space of $\,\tzm\nh$, not depending on
%the basis
$\,e_1\w,\dots,e_p\w$ since
\begin{equation}\label{hsp}
H\nh(E)\,=\,\{v\in\tzm\nnh:\mathrm{span}(v,E)\,\mathrm{\ is\ an\ integral\ element\
of\ }\,\mathcal{I}\}\hh.
\end{equation}
For fixed real constants $\,\ve,\tz,\kappa$, let the open subset
$\,\mathcal{Y}\hs$ of $\,\bbR\nh^6$
%\nh$, with the coordinates $\,\vt,\lambda,\jq,\js,\jp\nh,\jr$,
consist of all points 
%$\,\mathbf{y}=
$\,(\vt,\lambda,\jq,\js,\jp\nh,\jr)\in\bbR\nh^6$ such that
$\,\jq\,$ and $\,\jq\nh\jp-\js^2$ are both positive, while $\,\vt\,$ lies in 
the domains of $\,\az\,$ and $\,F\,$ chosen so as to satisfy (\ref{tap}). 
Consider now
\begin{equation}\label{ids}
\begin{array}{l}
\mathrm{the\ exterior\ differential\ system\ 
}\,\,\mathcal{I}\,\hh\hs\mathrm{ \ on\ this\ 
}\,\hh\hs\mathcal{Y}\,\mathrm{\ generated}\\
\mathrm{by\ the\ two\ }\,1\hyp\mathrm{forms\ }\,\,d\jr\,+\,\js\nh\az'd\vt\,
-\,(\jq\az'\nh+\hs F')\,d\lambda\,\,\mathrm{\ and}\\
d\hh(\jq\nh\jp\nnh-\nnh\js^2)\nnh-\nnh[(\jq\nh\jp\nnh-\nnh\js^2)\az\nnh
+\nnh\jp F\nnh-\nnh\js\jr]\hs d\vt\nnh
-\nnh(\jq\jr\nnh-\nnh\js F)\hs d\lambda\hh,\\
\mathrm{the\ }\,\hs2\hyp\mathrm{form\ \ }\,d\jq\wedge d\lambda\,\,
+\,\,d\vt\wedge\hh d\js\,\,-\,\,(\jq\az\hs+\,F)\,d\vt\wedge d\lambda\hh,\\
\mathrm{the\ \ }\,\hs2\hyp\mathrm{form\ \ \ }\hh\,d\js\wedge d\lambda\hs\,
+\hs\,d\vt\hs\wedge d\jp\hs\,
-\,(\js\nh\az\hs+\hs\jr)\,d\vt\wedge d\lambda\hh,\\
\mathrm{their\hn\ exterior\hn\ derivatives,\nh\ and\hn\ the\hn\ exterior\hn\ 
derivatives\hn\ of}\\
%\mathrm{of\ %the\ }\,1\hyp\mathrm{forms\ 
(\jq\nh\jp\nnh-\nnh\js^2)^{-\nnh1}\nh(\jp\hs d\vt\nh-\nh\js\hs d\lambda)\,
\mathrm{\ and\ }\,(\jq\nh\jp\nnh-\nnh\js^2)^{-\nnh1}\nh(\js\hs d\vt\nh
-\nh\jq\,d\lambda)\hn.
\end{array}
\end{equation}
The choice of $\,\mathcal{I}\hs$ is justified as follows. We want
$\,\mathcal{I}\hs$ to consist of differential forms that are expected
to vanish on all graphs of solutions to (\ref{foe}) with (\ref{det}). Since
such a graph is 
horizontal (in the sense that $\,d\vt\wedge d\lambda\ne0\,$ on it), forms
generating $\,\mathcal{I}\hs$ will result from multiplying 
by $\,d\vt\wedge d\lambda\,$ the left-hand sides in (\ref{foe}),
along with those in its consequence (\ref{qbt}), 
as including the latter changes nothing except for enriching the differential
system. Each of the terms $\,\varXi\jz\nh_\svt\w$, or
$\,\varXi\jz\hskip-2.2pt_\lambda\w$, or $\,\varTheta$, where
$\,\jz\in\{\jq,\js,\jp\nh,\jr\}\,$ (with various $\,\varXi\,$ depending on
$\,\jq,\js,\jp\nh,\jr$, and $\,\varTheta\,$ which may further depend also on
$\,\az,F\nh,\az'\nh,F'$) then becomes
$\,\varXi\,d\jz\wedge d\lambda$, or $\,\varXi\,d\vt\wedge d\jz$, or 
simply $\,\varTheta\,d\vt\wedge d\lambda$.

The $\,2$-forms in the fourth and fifth lines of (\ref{ids}) arise as
above from (\ref{foe}.a) and (\ref{foe}.b), the exterior
derivatives of the $\,1$-forms in the last line -- from (\ref{foe}.c) and
(\ref{foe}.d) via (\ref{php}). The
rest of (\ref{ids}) is based on an additional principle: if our exterior 
differential system ends up containing $\,\xi\wedge d\vt\,$ and
$\,\xi\wedge d\lambda$, for some $\,1$-form $\,\xi$, we are free to
include $\,\xi\,$ among the
system's generators, as the horizontal integral sur\-faces/planes then
obviously remain unaffected. The $\,1$-form $\,\xi\,$ in the second, or third,
line of (\ref{ids}) corresponds in this way to (\ref{foe}.e) -- (\ref{foe}.f)
or, respectively, (\ref{qbt}).

Instead of invoking the general ``additional principle'' one can also justify
the second and third lines in (\ref{ids}) directly from the fact that
(\ref{dqb}) is a consequence of (\ref{foe}), which causes the two $\,1$-forms
to vanish on all graphs of solutions to (\ref{foe}).

\section{The unique\hs-extension theorem}\label{ue}
\setcounter{equation}{0}
In $\,\bbR\nh^6$ with the coordinates $\,\vt,\lambda,\jq,\js,\jp\nh,\jr$, 
given a sub\-space $\,E\subseteq\bbR\nh^6\nh$,
\begin{equation}\label{hrz}
\mathrm{we\ call\ }\,E\,\mathrm{\ horizontal\ when\
}\,(d\vt,d\lambda)\nnh:\nnh E\hs\to\rto\mathrm{\ is\ injective.}
\end{equation}
\begin{remark}\label{unqhr}If $\,E_1\w$ is an integral line of 
$\,\mathcal{I}\hs$ in (\ref{ids}) and $\,E_1\w\subseteq E_2\w$ for 
a unique horizontal integral plane $\,E_2\w$, then $\,E_2\w$ is 
the only integral plane containing $\,E_1\w$. Namely, another such plane 
$\,E_2'$, being non\-hor\-i\-zon\-tal, would intersect the kernel of
$\,(d\vt,d\lambda)\,$ along a line. As 
$\,E_3\w\subseteq H\nh(E_1\w)\,$ for the vector sub\-space $\,H\nh(E)\,$ in
(\ref{hsp}) and the three-di\-men\-sion\-al span $\,E_3\w$ of
$\,E_2\w$ and $\,E_2'$, all planes in $\,E_3\w$ containing the line $\,E_1\w$,
other than $\,E_2'$, would be horizontal integral planes, making $\,E_2\w$
nonunique.
\end{remark}
\begin{remark}\label{simkr}The only $\,1$-forms in $\,\mathcal{I}\hs$ are,
obviously, the functional combinations of those in the second and third 
lines of (\ref{ids}). Due to their linear independence at every point, the
simultaneous kernel of these two $\,1$-forms is a co\-dimen\-sion-two 
distribution $\,\mathcal{D}\,$ on $\,\mathcal{Y}\nh$, that is, a vector
sub\-bun\-dle of $\,T\hn\mathcal{Y}\nh$, and its fibre at any 
$\,(\vt,\lambda,\jq,\js,\jp\nh,\jr)\in\mathcal{Y}\,$ consists of all
$\,(\dot\vt,\dot\lambda,\dot\jq,\dot\js,\dot\jp\nh,\dot\jr)\in\bbR\nh^6$
with (\ref{dge}). Hence
\begin{equation}\label{cha}
\begin{array}{l}
\mathrm{vectors\ 
}\hs(\dot\vt,\dot\lambda,\dot\jq,\dot\js,\dot\jp\nh,\dot\jr)\hs
\mathrm{\ at\ }\hs(\vt,\lambda,\jq,\js,\jp\nh,\jr)\hs\mathrm{\ spanning\ 
horizontal}\\
\mathrm{integral\ lines\ of\ }\,\,\mathcal{I}\,\mathrm{\ are\ 
characterized\ by\ (\ref{dge})\ and\ }\,(\dot\vt,\dot\lambda)\hs\ne(0,0),
\end{array}
\end{equation}
where $\,\az,F\nh,\az'\nh,F'$ satisfy (\ref{tap}) -- (\ref{rng}).

\end{remark}
\begin{theorem}\label{exone}Every horizontal integral line of the system\/ 
$\,\mathcal{I}\hs$ defined by\/ {\rm(\ref{ids})} is contained in a
unique integral plane of\/ $\,\mathcal{I}\nh$, and this unique plane
is also horizontal.
\end{theorem}
\begin{proof}Any horizontal plane in $\,\bbR\nh^6$ has, by
(\ref{hrz}), a unique basis of the form
\begin{equation}\label{bas}
(1,0,\jq\hn_\svt\w,\js\nh_\svt\w,\jp\hskip-1.3pt_\svt\w,\jr\nh_\svt\w)\hh,\quad
(0,1,\jq\nh_\lambda\w,\js\nh_\lambda\w,\jp\hskip-1.3pt_\lambda\w,
\jr\nh_\lambda\w)
\end{equation}
The span of (\ref{bas}) is an {\it integral plane of\/}
$\,\mathcal{I}\hs$
if and only if all the $\,1$-forms (and, $\,2$-forms) listed in (\ref{ids}) 
yield the value $\,0\,$ when evaluated on both vectors in (\ref{bas}) or,
respectively, on the pair (\ref{bas}). Due to the two final paragraphs of
Section~\ref{ae}, and (\ref{php}), this is equivalent to (\ref{foe}) and 
(\ref{qbt}), and hence (Lemma~\ref{cnseq}) just to (\ref{foe}).

Every vector in (\ref{cha}) is a linear combination of a unique pair
(\ref{bas}) satisfying (\ref{foe}): as the coefficients of the combination
must be $\,\dot\vt$ and $\,\dot\lambda$, this is immediate from
Theorem~\ref{affsb}. In other words, every horizontal integral line of
$\,\mathcal{I}\hs$ lies within a unique horizontal integral plane.
Remark~\ref{unqhr} now allows us to drop the last occurrence of the word
`horizontal', completing the proof.
\end{proof}

\section{Existence of integral surfaces}\label{ei}
\setcounter{equation}{0}
The next fact -- used below to derive our Theorem~\ref{exist} --
is a special case of the celebrated Car\-tan-K\"ah\-ler 
theorem \cite[pp.\,81--82]{bryant-chern-gardner-goldschmidt-griffiths}. Since
our phrasing differs from that of 
\cite{bryant-chern-gardner-goldschmidt-griffiths}, we devote the next section
to clarifying how our version amounts to adapting the one in
\cite{bryant-chern-gardner-goldschmidt-griffiths} to our
particular case.

The symbols $\,\mathcal{Y}\nh,\mathcal{I}\hs$ and $\,\mathcal{D}\,$ stand
here for more general objects that those in Sections~\ref{ae}--\ref{ue}. 
The definition 
(\ref{hrz}) of horizontality, for integral elements, is used more generally,
as well as extended, in an obvious fashion, to integral manifolds.
\begin{theorem}\label{cakae}Let real-an\-a\-lyt\-ic functions\/
$\,\vt,\lambda\,$ and\/ $1$-forms\/ $\,\xi_1\w,\dots,\xi_q\w$ on a
manifold\/ $\,\mathcal{Y}\nnh$, where\/ $\,0<\hs q<\dim\mathcal{Y}\nnh$, have
the property that\/ $\,d\vt,d\lambda,\,\xi_1\w,\dots,\xi_q\w$ are
linearly independent at every point. Denoting by\/ $\,\mathcal{D}\hs$ and\/ 
$\,\mathcal{I}$ the distribution on\/ $\,\mathcal{Y}\nh$ arising as the 
simultaneous kernel of the\/ $1$-forms\/ $\,\xi_1\w,\dots,\xi_q\w$ and,
respectively, the exterior differential system on\/ $\,\mathcal{Y}\nh$ 
generated by\/ $\,\xi_1\w,\dots,\xi_q\w$ and, possibly, some high\-er-de\-gree
forms, along with their exterior derivatives, let us suppose that
\begin{equation}\label{evh}
\begin{array}{l}
\mathrm{every\hskip1.15pt\ horizontal\hskip1.15pt\ integral\hskip1.15pt\
line\hskip1.15pt\ of\hskip1.15pt\
}\hskip1.15pt\mathcal{I}\nh\mathrm{,\hskip1.15pt\ at\hskip1.15pt\
any\hskip1.15pt\ point}\\
\mathrm{of\ }\mathcal{Y}\nnh\mathrm{,\nnh\ is\ contained\ in\ a\ unique\ 
integral\ plane\ of\ }\,\mathcal{I}\nh.
\end{array}
\end{equation}
Then every horizontal real-an\-a\-lyt\-ic integral curve of\/
$\,\mathcal{I}\hs$ is contained, locally, in a lo\-cal\-ly-u\-nique
horizontal real-an\-a\-lyt\-ic integral surface. 
Examples of such curves are provided by un\-pa\-ram\-e\-trized integral curves 
of any real-an\-a\-lyt\-ic vector field without zeros forming a horizontal
local section of the
vector bundle\/ $\,\mathcal{D}\hs$ over\/ $\,\mathcal{Y}\nnh$. Also,
\begin{equation}\label{igl}
\mathrm{integral\ lines\ of\ }\,\mathcal{I}\hs\mathrm{\ are\ the\ same\
as\ lines\ tangent\ to\ }\,\mathcal{D}.
\end{equation}
\end{theorem}
\begin{remark}\label{chdat}Due to Theorem~\ref{exone}, our 
$\,\mathcal{Y}\,$ and $\,\mathcal{I}\nh$, introduced in 
Section~\ref{ae}, satisfy the hypotheses of Theorem~\ref{cakae}, with 
$\,q=2$, the coordinate  functions $\,\vt,\lambda$, and the 
two $\,1$-forms in the second and third lines of (\ref{ids}). Therefore, 
our $\,\mathcal{D}$ (see Remark~\ref{simkr}) then corresponds to
$\,\mathcal{D}\,$ in Theorem~\ref{cakae}, and hence satisfies (\ref{igl}).
Horizontal integral curves of our $\,\mathcal{I}\hs$ thus are, by (\ref{cha}), 
precisely those curves which have parametrizations
$\,t\mapsto(\vt,\lambda,\zb)=(\vt,\lambda,\jq,\js,\jp\nh,\jr)\in\mathcal{Y}\,$
with (\ref{dge}), where $\,(\,\,)\dot{\,}=\,d/dt$, and
$\,(\dot\vt,\dot\lambda)\hs\ne(0,0)\,$ for all $\,t$. Choosing such a curve as
in the sentence preceding (\ref{igl}), we obtain the additional data
$\,\zb_\svt\w,\zb\nh_\lambda\w$ required in Theorem~\ref{exist} by
applying to $\,\dzb\,$ the inverse of the af\-fine iso\-mor\-phism 
$\,\varPhi:\mathcal{A}\to\mathcal{L}\,$ of Theorem~\ref{affsb}. This clearly
results in an in\-fi\-nite\hs-di\-men\-sion\-al freedom of choices mentioned
at the end of Section~\ref{sl}.
\end{remark}
\begin{proof}[Proof of Theorem~\ref{exist}]
The image of the mapping
$\,t\mapsto(\vt,\lambda,\zb)\in\mathcal{Y}\,$ in the second paragraph of 
Theorem~\ref{exist} is a horizontal
real-an\-a\-lyt\-ic integral curve of $\,\mathcal{I}\nh$. In fact,
horizontality follows since $\,t\mapsto(\vt,\lambda)\,$ is an embedding,
while, as $\,\dzb=\dot\vt\zb\hn_\svt\w+\dot\lambda\hs\zb\nh_\lambda\w$, 
the resulting tangent directions are integral lines of $\,\mathcal{I}\hs$
as an immediate consequence of Theorem~\ref{exone} combined with (\ref{cha}).

The integral surface of $\,\mathcal{I}\hs$ arising in Theorem~\ref{cakae},
being horizontal (Theorem~\ref{exone}), forms, locally, the graph of a
function $\,(\vt,\lambda)\mapsto\zb=(\jq,\js,\jp\nh,\jr)$, which is
a solution to (\ref{foe}) according to the description of $\,\mathcal{I}\hs$
in (\ref{ids}) and the paragraph following (\ref{ids}). To realize the
condition $\,\jq\nh_\lambda\w\nh\ne0\,$ required in (\ref{det}) we solve
(\ref{foe}), at a given point
$\,(\vt,\lambda,\jq,\js,\jp\nh,\jr)\in\mathcal{Y}\nh$, by setting
$\,(\jq\hn_\svt\w,\jq\nh_\lambda\w)=(0,1)$, which uniquely determines
$\,\js\nh_\svt\w,\jp\hskip-1.3pt_\svt\w,\jr\nh_\svt\w,
\js\nh_\lambda\w,\jp\hskip-1.3pt_\lambda\w,\jr\nh_\lambda\w$, and then
choosing the quadruple (\ref{dqd}), with fixed
$\,(\dot\vt,\dot\lambda)\hs\ne(0,0)$, associated with
the resulting octuple (\ref{qts}).
\end{proof}
Under the assumptions of Theorem~\ref{cakae},
let $\,k=\dim\mathcal{Y}\nh$. For 
all $\,p\hs$-di\-men\-sion\-al horizontal integral elements $\,E
=E_p\w$ of 
$\,\mathcal{I}\nh$, with $\,p\in\{0,1\}$, and for 
$\,r(E)=\dim H\nh(E)-(p+1)\,$ in (\ref{hee}), 
\begin{equation}\label{fxd}
\begin{array}{rl}
\mathrm{a)}&\mathrm{the\ integer\ }\,r(E)\,\mathrm{\ has\ a\ fixed\
nonnegative\ value,\  namely,}\\
\mathrm{b)}&
\dim H\nh(E_0\w)=\,k-\hs q\,\mathrm{\ and\ }\hn\,r(E_0\w)
=\,k-\hs q-1\,\nh\mathrm{\ when\ }\,p=0,\\
\mathrm{c)}&\dim H\nh(E_1\w)\hn=2\,\,\mathrm{\ and\ }\,\,r(E_1\w)\hn
=0\,\,\mathrm{\ in\ the\ case\ where\ }\,p=1\hn. 
\end{array}
\end{equation}
This is obvious from (\ref{igl}) or, respectively, (\ref{evh}).

\section{Where Theorem~\ref{cakae} comes from}\label{wt}
\setcounter{equation}{0}
Here is the Car\-tan-K\"ah\-ler theorem, cited {\it verbatim\/} from 
\cite[pp.\,81--82]{bryant-chern-gardner-goldschmidt-griffiths}:
%, with notation different in part from ours:

{\it Let\/ $\,\mathcal{I}\subset\hn\varOmega\hh^*\nnh(M)\,$ be a
real an\-a\-lyt\-ic differential ideal. Let\/ $\,P\subset\nh M\,$ be a
connected,
$\,p\hs$-di\-men\-sion\-al, real an\-a\-lyt\-ic, K\"a\-hler-reg\-u\-lar integral
manifold of\/ $\,\mathcal{I}\nh$.

Suppose that\/ $\,r=r(P)\,$ is a non-neg\-a\-tive integer. Let\/
$\,R\subset\nh M\,$ be a real an\-a\-lyt\-ic  sub\-man\-i\-fold of\/ $\,M\,$
which is of 
co\-dimen\-sion\/ $\,r$, which contains\/ $\,P\nh$, and which satisfies the
condition that\/ $\,\txr\,$ and\/ $\,H(\txr)\,$ are transverse in\/ $\,\txm\,$ 
or all\/ $\,x\in P\nh$.

Then there exists a real analytic integral manifold of\/ $\,\mathcal{I}\nh$,
$\,X$, which is connected and\/ $\,(p+1)$-di\-men\-sion\-al and which
satisfies\/ $\,P\subset X\subset R$. This manifold is unique in the sense that
any other real analytic integral manifold of\/ $\,\mathcal{I}\hs$ with these
properties agrees with\/ $\,X\,$ on an open neighborhood of\/ $\,P\nh$.}

As we verify in the following paragraphs, 
the hypotheses of our Theorem~\ref{cakae} imply those listed above, for
$\,(p,r)=(1,0)$, the manifolds $\,M\nh,R\,$ above which are both replaced
by our $\,\mathcal{Y}\nnh$, and the same ideal $\,\mathcal{I}\hs$ as ours.
By our $\,\mathcal{Y}\hs$ and $\,\mathcal{I}\hs$ we mean the ``general'' ones
(see the three lines preceding Theorem~\ref{cakae}), rather than the very
special choices of $\,\mathcal{Y}\hs$ and $\,\mathcal{I}\hs$ made in
Section~\ref{ae}.

Furthermore, $\,P\hs$ mentioned above is our (arbitrary) horizontal
real-an\-a\-lyt\-ic integral curve of $\,\mathcal{I}\hs$. The resulting
manifold $\,X\,$ corresponds to the horizontal real-an\-a\-lyt\-ic integral
surface of $\,\mathcal{I}\hs$ claimed to exist in
Theorem~\ref{cakae}.

We now proceed to explain why our horizontal integral curve must automatically
be K\"a\-hler-reg\-u\-lar 
\cite[p.\,81]{bryant-chern-gardner-goldschmidt-griffiths}, meaning that its
tangent lines are all K\"a\-hler-reg\-u\-lar in the sense of
\cite[p.\,68,\,Definition 1.7]{bryant-chern-gardner-goldschmidt-griffiths}.
To verify this last claim, we first apply Car\-tan's test 
\cite[p.\,74,\,Theorem 1.11]{bryant-chern-gardner-goldschmidt-griffiths}.
Namely, in the notation of
\cite[p.\,74,\,Theorem 1.11]{bryant-chern-gardner-goldschmidt-griffiths},
$\,n=1\,$ (as we are dealing with tangent {\it lines}). Due to
the relation $\,\dim H\nh(E_0\w)=k-q\,$ in (\ref{fxd}.b), and
(\ref{igl}), $\,H\nh(E_0\w)\,$ is of co\-dimen\-sion $\,q$ in the tangent
space of $\,\mathcal{Y}\hs$ containing it, the same as the co\-dimen\-sion, in
the $\,(2k-1)$-di\-men\-sion\-al Grass\-mann manifold
$\,\mathrm{Gr}_1\w\hs\mathcal{Y}\,$ of lines tangent
to $\,\mathcal{Y}\nh$, of the $\,(2k-q-1)$-di\-men\-sion\-al sub\-man\-i\-fold
$\,V\hskip-3pt_1\w(\mathcal{I})\,$ formed by all integral lines of
$\,\mathcal{I}\nh$. Car\-tan's test thus shows that every line $\,E_1\w$
tangent to our horizontal integral curve is {\it ordinary\/} 
\cite[p.\,73,\,Definition 1.9]{bryant-chern-gardner-goldschmidt-griffiths}. 
The K\"a\-hler-reg\-u\-lar\-i\-ty of $\,E_1\w$ now trivially follows, as
$\,r\,$ in \cite[pp.\,67-68]{bryant-chern-gardner-goldschmidt-griffiths} 
has the constant value $\hs0\hs$ according to (\ref{fxd}.c). This is 
also the value $\,r=r(P)\,$ in the italicized statement cited above from 
\cite{bryant-chern-gardner-goldschmidt-griffiths}. Cf.\
\cite[pp.\,81--82, the lines preceding
Theorem 2.2]{bryant-chern-gardner-goldschmidt-griffiths}.

\section{Proof of Theorem~\ref{reali}}\label{pt}
\setcounter{equation}{0}
Let $\,\nabla\hs$ (or, $\,g$) be a connection (or, a
pseu\-\hbox{do\hs-}Riem\-ann\-i\-an metric) on a $\,C^\infty$ manifold
$\,M\nh$. We call $\,\nabla\hs$ or $\,g\,$ {\it real-an\-a\-lyt\-ic\/} if, in
a suitable coordinate system around every point of $\,M\nh$, its components 
$\,\vg_{\hskip-2.1ptjk}^{\hs l}$ (or, $\,g_{jk}\w$) are real-an\-a\-lyt\-ic
functions of the coordinates. The $\,C^\infty$ structure of $\,M\,$ then
contains a unique real-an\-a\-lyt\-ic structure (maximal atlas) making
$\,\nabla\hs$ or, $\,g\,$ real-an\-a\-lyt\-ic. (The atlas consists of all
coordinate systems just mentioned; their mutual transition mappings are
real-an\-a\-lyt\-ic due to real-an\-a\-lyt\-ic\-i\-ty of af\-fine mappings,
or isometries, between manifolds with real-an\-a\-lyt\-ic
con\-nec\-tions/met\-rics, which follows since such mappings appear linear in
geodesic coordinates.) Real-an\-a\-lyt\-ic\-i\-ty of a metric $\,g\,$
obviously implies that of its Le\-vi-Ci\-vi\-ta connection $\,\nabla\hs$
(and vice versa, since $\,\nabla\nh g=0$).

For a real-an\-a\-lyt\-ic (Riemannian) K\"ah\-ler metric $\,g\,$ on a complex
manifold $\,M\nh$, the unique real-an\-a\-lyt\-ic structure described above
coincides with the one induced by the complex structure of $\,M\nh$. In fact, 
local hol\-o\-mor\-phic coordinate functions, being $\,g$-har\-mon\-ic, must
be real-an\-a\-lyt\-ic relative to the former structure, as a consequence of
the standard regularity theory of elliptic partial differential equations
applied to the $\,g$-La\-plac\-i\-an $\,\Delta$.
\begin{proof}[Proof of Theorem~\ref{reali}]Combining Theorems~\ref{exist}
and~\ref{equiv}, we obtain the first assertion of Theorem~\ref{reali}.

For the second one we invoke the existence results of \cite{wang-zhu} and 
\cite{chen-lebrun-weber}. In both cases, $\,dQ\wedge d\vt\ne0\,$ 
somewhere, and the metric is real-an\-a\-lyt\-ic. The former claim follows,
for instance, since a compact K\"ah\-ler surface with a nontrivial
hol\-o\-mor\-phic gradient $\,\nabla\nh\vt\,$ having $\,dQ\wedge d\vt=0\,$
identically for $\,\jq=g(\nabla\nh\vt,\nnh\nabla\nh\vt)\,$ must
necessarily \cite[Sect.\,1]{derdzinski-iu} be 
bi\-hol\-o\-mor\-phic to $\,\bbCP^2$ or a $\,\bbCP^1$ bundle over $\,\bbCP^1$
(rather than the two-point blow-up of $\,\bbCP^2$). The latter, in the case
of \cite{wang-zhu}, is due to a general reason: all Ric\-ci sol\-i\-tons are 
real-an\-a\-lyt\-ic \cite[Lemma\,3.2]{dancer-wang}. So are, however,
all Riemannian Ein\-stein metrics \cite[Theorem\,5.2]{deturck-kazdan}, and
the Chen-LeBrun-Weber metric of \cite{chen-lebrun-weber} is con\-for\-mal 
to an Ein\-stein metric $\,\hatg$, while again, for a general reason
\cite[p.\,417, Prop.\,3(ii)]{derdzinski-83}, the con\-for\-mal change leading
from $\,\hatg\,$ to $\,g\,$ has a canonical form (up to a constant factor, it
is the multiplication by the cubic root of the norm-squared of the self-dual
Weyl tensor). This causes $\,g\,$ to be real-an\-a\-lyt\-ic as well.
\end{proof}

\section{The an\-a\-lyt\-ic-con\-tin\-u\-a\-tion phenomenon}\label{ac}
\setcounter{equation}{0}
We elaborate here on the plausibility of small 
deformations mentioned in the lines following Theorem~\ref{reali}, 
beginning with the coth-cot an\-a\-lyt\-ic con\-tin\-u\-a\-tion. 
The real-an\-a\-lyt\-ic function $\,\bbR\ni y\mapsto y^{-\nnh1}\nnh\tanh y$,
with the value $\,1\,$ at $\,y=0$, being even, has the form $\,\varSigma(y^2)\,$ for
some real-an\-a\-lyt\-ic  function $\,\varSigma$. Now 
$\,(\ve,\vt)\mapsto\bz\hn_\ve\w\hn(\vt)=\vt \varSigma(\ve\vt^2)\,$ is a
real-an\-a\-lyt\-ic function on an open subset of $\,\rto$ and
$\,\bz\hn_\ve\w\hn(\vt)$ equals $\,\ve^{-\nnh1/2}\tanh(\ve^{1/2}\vt)$, or
$\,\vt$,
or $\,|\ve|^{-\nnh1/2}\tan(|\ve|^{1/2}\vt)$, depending on whether $\,\ve>0$,
or $\,\ve=0$, or $\,\ve<0$. For
$\,\az\hn_\ve\w(\vt)=2/\nnh\bz\hn_\ve\w\hn(\vt)\,$
the analogous expressions are
\[
2\hs\ve^{1/2}\coth(\ve^{1/2}\vt)\,\,\,\,(\mathrm{if\ }\,\ve>0)\hh,\hskip14pt
2/\nh\vt\,\,\,\,(\mathrm{if\ }\,\ve=0)\hh,\hskip14pt
2\hs|\ve|^{1/2}\cot(|\ve|^{1/2}\vt)\,\,\,\,(\mathrm{if\ }\,\ve<0)\hh.
\]
All $\,\az\hn_\ve\w$ with $\,\ve>0$, as well as those with $\,\ve<0$, are thus
af\-fine (in fact, linear) modifications -- see Remark~\ref{affmo} -- of
$\,\az_1\w$ or,
respectively, $\,\az\nh_{-\nnh1}\w$, and $\,\az_0\w(\vt)=2/\nh\vt$.

For a tanh-coth an\-a\-lyt\-ic-con\-tin\-u\-a\-tion argument we define
$\,(t,\vt)\mapsto\az_t\w\hn(\vt)\,$ by
$\,\az\nh_t\w\hn(\vt)
=2(e^\evt\nh-\hs te^{-\nh\evt})/(e^\evt\nh+\hs te^{-\nh\evt})$.
Thus, with $\,q\,$ such that $\,2q=\log|t|$, if $\,t>0\,$ (or, $\,t<0$),
$\,\az\hn_t\w\hn(\vt)=2\hn\tanh(\vt-q)\,$ or,
respectively, $\,\az\hn_t\w\hn(\vt)=2\hn\coth(\vt-q)$. Again, all $\,\az_t\w$
for $\,t>0$,
or those with $\,t<0$, are af\-fine (this time, translational) modifications
of $\,\az_1\w$, or of $\,\az\nh_{-\nnh1}\w$, while $\,\az_0\w(\vt)=2$.

\end{document}